\documentclass[12pt]{amsart}
\usepackage{fullpage,amssymb,verbatim,mathrsfs,url}
\newtheorem{theorem}{Theorem}
\newtheorem{lemma}[theorem]{Lemma}
\newtheorem{proposition}[theorem]{Proposition}
\newtheorem{corollary}[theorem]{Corollary}
\newtheorem{definition}[theorem]{Definition}
\newtheorem{conjecture}[theorem]{Conjecture}
\theoremstyle{remark}
\newtheorem{remarks}[theorem]{Remarks}
\newtheorem{remark}[theorem]{Remark}
\newtheorem{examples}[theorem]{Examples}
\numberwithin{theorem}{section}
\numberwithin{equation}{section}
\newcommand{\A}{\mathbb{A}}
\newcommand{\C}{\mathbb{C}}
\newcommand{\R}{\mathbb{R}}
\newcommand{\Q}{\mathbb{Q}}
\newcommand{\Z}{\mathbb{Z}}
\newcommand{\LL}{\mathscr{L}}
\newcommand{\Sel}{\mathcal{S}}
\newcommand{\Frob}{\mathrm{Frob}}
\DeclareMathOperator{\ord}{ord}
\DeclareMathOperator{\GL}{GL}
\DeclareMathOperator{\supp}{supp}
\DeclareMathOperator{\Res}{Res}

\DeclareMathOperator{\Gal}{Gal}

\begin{document}
\title{$L$-functions as distributions}
\author{Andrew R.~Booker}
\address{School of Mathematics, University of Bristol,
University Walk, Bristol, BS8 1TW, United Kingdom}
\email{andrew.booker@bristol.ac.uk}
\thanks{The author was supported by EPSRC Fellowship EP/H005188/1.}
\begin{abstract}
We define an axiomatic class of $L$-functions extending the Selberg class.
We show in particular that one can recast the traditional conditions
of an Euler product, analytic continuation and functional equation in
terms of distributional identities akin to Weil's explicit formula. The
generality of our approach enables some new applications; for instance,
we show that the $L$-function of any cuspidal automorphic representation
of $\GL_3(\A_\Q)$ has infinitely many zeros of odd order.
\end{abstract}
\maketitle
\section{Introduction}
In \cite{selberg}, Selberg introduced his eponymous class of
$L$-functions, defined as follows.
\begin{definition}\label{def:selbergclass}
The \emph{Selberg class} $\Sel$ is the set of functions
$F$ satisfying the following axioms:
\begin{enumerate}
\item (Dirichlet series). There are numbers $a(n)\in\C$ such that
$F(s)=\sum_{n=1}^\infty a(n)n^{-s}$, converging absolutely for
$\Re(s)>1$.
\item (Analytic continuation). There is an integer $m\ge0$ such that
$(s-1)^mF(s)$ continues to an entire function of finite order.
\item (Functional equation).
There exist $k\in\Z_{\ge0}$, $Q,\lambda_1,\ldots,\lambda_k\in\R_{>0}$,
$\mu_1,\ldots,\mu_k\in\C$ with $\Re(\mu_j)\ge0$ and
$\epsilon\in\C$ with $|\epsilon|=1$ such that the function
$$
\Phi(s)=\epsilon Q^s\prod_{j=1}^r\Gamma(\lambda_j{s}+\mu_j)\cdot F(s)
$$
satisfies the functional equation
$$
\Phi(s)=\overline{\Phi(1-\bar{s})}.
$$
\item (Ramanujan hypothesis).
For every $\varepsilon>0$, $a(n)\ll_\varepsilon n^\varepsilon$.
\item (Euler product). $a(1)=1$ and
$\log{F}(s)=\sum_{n=2}^\infty b(n)n^{-s}$,
where $b(n)$ is supported on prime powers, and
$b(n)\ll n^\theta$ for some $\theta<\frac12$.
\end{enumerate}
\end{definition}

Selberg went on to pose various conjectures about the elements of
$\Sel$, in particular:
\begin{conjecture}[Selberg orthogonality conjecture]
\label{conj:selberg}\hspace{1mm}
Let $F,G\in\Sel$ be primitive, in the sense that they
cannot be expressed non-trivially as products of elements of $\Sel$,
with Dirichlet coefficients $a_F(n)$ and $a_G(n)$.
Then
$$
\sum_{\substack{p\text{ prime}\\p\le x}}
\frac{a_F(p)\overline{a_G(p)}-\delta_{FG}}p
\ll_{F,G}1,
$$
where $\delta_{FG}=1$ if $F=G$ and $0$ otherwise.
\end{conjecture}
This idea of codifying the properties of $L$-functions is appealing
as an alternative to the Langlands program. However, one immediate
problem is that it is not obvious which properties of $L$-functions
should be taken as axioms, and which are theorems to be derived from
the axioms.  More to the point, Selberg's choice of axioms does not
correspond perfectly to the properties of the known $L$-functions, i.e.\
those associated to automorphic representations. For instance, the
Ramanujan bound $a(n)\ll_\varepsilon n^\varepsilon$ remains a conjecture
for most automorphic $L$-functions, but Conjecture~\ref{conj:selberg}
is essentially known in that context.\footnote{In full generality it is
known in a slightly weaker form that includes the prime powers in the
sum, and those may be removed under a mild hypothesis; see \cite{wy,as}.}
This difficulty seems inherent to the axiomatic approach and may never
be resolved completely, since there are differing opinions about what
properties of $L$-functions are essential, and it is likely impossible
to avoid making at least some choices based purely on aesthetics.

Nevertheless, Selberg's paper has been influential in shaping the way
that researchers think about $L$-functions, and has spurred a large
volume of research, both attempting to classify the elements of $\Sel$
and studying the consequences of Conjecture~\ref{conj:selberg}.
The general belief is that $\Sel$ essentially coincides with the
class of automorphic $L$-functions. However, Selberg's list of axioms
is in principle more flexible; for instance, the local Euler factor
$\exp\bigl(\sum_{n=1}^\infty b(p^n)p^{-ns}\bigr)$ can be any function of
the form $e^{f(p^{-s})}$, where $f$ is analytic on a disc of radius
$p^{-\theta}$ for some $\theta<\frac12$ and satisfies $f(0)=0$.
This is substantially more general than the factors that occur for
automorphic $L$-functions (which are always reciprocal polynomials
of $p^{-s}$) and permits some natural operations, such as taking
square roots and quotients.  On the other hand, the $\Gamma$-factors
$\Gamma(\lambda{s}+\mu)$, while again more general than those associated
to automorphic $L$-functions (for which we may always reduce to
the case $\lambda=\frac12$), do not seem to occur naturally when
$\lambda\notin\frac12\Z_{>0}$, so this offers no effective increase in
generality.\footnote{Selberg acknowledges in a footnote of his paper
that we may take $\lambda$ to be a half-integer in every known case.
It seems likely that he did not intend for his definition to be taken
as a serious attempt at generalization, but rather as a recognition
of the fact that the $\Gamma$-factors are non-canonical because of
the Legendre and Gauss multiplication formulas. We note that those
identities have analogues at the finite places as well, e.g.\ the
Legendre duplication formula is analogous to the ``difference of squares''
identity $1-p^{-2s}=(1-p^{-s})(1+p^{-s})$, but this ambiguity causes no
real confusion.  Note also that the analogue of $\Gamma(\lambda{s})$ is
the generalized Dirichlet series $1/(1-p^{-2\lambda{s}})$, which is not
permitted under Selberg's definition unless $\lambda$ is a half-integer.}

In this paper, we propose a broader set of axioms with the goal
of putting the $\Gamma$-factors on the same level of generality as the
other Euler factors, and as we will show, this enables some new
applications.  Our approach is to change language, and express
everything not in terms of $L$-functions directly (since there is no
agreement on how they should be defined anyway), but in terms of their
\emph{explicit formulae}.
Following the point of view introduced by Weil \cite{weil1}, these are
identities of distributions relating the zeros of an $L$-function to the
coefficients of its logarithmic derivative via the Fourier transform. For
instance, if $\chi\pmod{q}$ is an even primitive Dirichlet character
with complete $L$-function $\Lambda(s,\chi)=\Gamma_\R(s)L(s,\chi)$
and $g:\R\to\C$ is a sufficiently nice test function (e.g.\ smooth of
compact support) with Fourier transform $h(z)=\int_\R g(x)e^{izx}\,dx$
satisfying $h(\R)\subseteq\R$, then the explicit formula is the identity
$$
\begin{aligned}
\sum_{z\in\C}m(z)h(z)=2\Re\biggl[
&\int_0^\infty\bigl(g(0)-g(x)\bigr)\frac{e^{-x/2}}{1-e^{-2x}}\,dx\\
&+\frac12\left(\log\frac{q}{8\pi}-\gamma-\frac{\pi}2\right)
g(0)-\sum_{n=2}^{\infty}
\frac{\Lambda(n)\chi(n)}{\sqrt{n}}g(\log{n})\biggr],
\end{aligned}
$$
where $m(z)=\ord_{s=\frac12+iz}\Lambda(s,\chi)$.

Here the integral kernel $\frac{e^{-x/2}}{1-e^{-2x}}$ is related to the
$\Gamma$-factor $\Gamma_\R(s)$ associated to $\chi$ (by a logarithmic
derivative and Fourier transform).  Since the explicit formula is
additive, i.e.\ the formula for a product of $L$-functions is the
sum of the individual formulas, in this language it is clear how the
$\Gamma$-factors can be deformed.  For instance, replacing $\Gamma_\R(s)$
by its square root amounts to dividing the kernel by $2$.  It is also
now clear how to generalize the notion of $\Gamma$-factor---we simply
consider any suitable integral kernel. Of course, which kernel functions
should be considered ``suitable'' is again open to interpretation, but
there is one essential feature of the kernels occurring in the explicit
formulae of $L$-functions that must be present, namely a first-order
singularity at $0$, with residue reflecting the degree.\footnote{It is
tempting to consider more general singularities as well, but we
quickly find ourselves in a much larger landscape of functions that is
presumably very hard to classify; for instance, the Selberg zeta-functions
and their trace formulae give identities of this type with second-order
singularities.} All other conditions should be chosen to suit the
desired applications. With that in mind, after some trial and error,
we arrived at the following definition.
\begin{definition}\label{def:Ldatum}
An \emph{$L$-datum} is a triple $F=(f,K,m)$, where $f:\Z_{>0}\to\C$,
$K:\R_{>0}\to\C$ and $m:\C\to\R$ are functions satisfying the
following axioms:
\begin{itemize}
\item[(A1)]$f(1)\in\R$, $f(n)\log^k{n}\ll_k1$ for all $k>0$, and
$\sum_{n\le x}|f(n)|^2\ll_\varepsilon x^\varepsilon$ for all
$\varepsilon>0$;
\item[(A2)]$xK(x)$ extends to a Schwartz function on $\R$, and
$\lim_{x\to0^+}xK(x)\in\R$;
\item[(A3)]$\supp(m)=\{z\in\C:m(z)\ne0\}$ is discrete and contained in a
horizontal strip $\{z\in\C:|\Im(z)|\le y\}$ for some $y\ge0$,
$\sum_{\substack{z\in\supp(m)\\|\Re(z)|\le T}}|m(z)|\ll 1+T^A$
for some $A\ge0$, and
$\#\{z\in\supp(m):m(z)\notin\Z\}<\infty$;
\item[(A4)]for every smooth function $g:\R\to\C$
of compact support and Fourier transform
$h(z)=\int_{\R}g(x)e^{ixz}\,dx$ satisfying
$h(\R)\subseteq\R$, we have the equality
$$
\sum_{z\in\supp(m)}m(z)h(z)
=2\Re\left[\int_0^\infty{K(x)}(g(0)-g(x))\,dx
-\sum_{n=1}^{\infty}f(n)g(\log n)\right].
$$
\end{itemize}
Given an $L$-datum $F=(f,K,m)$,
we associate an \emph{$L$-function} $L_F(s)$ defined by
$$
L_F(s)=\sum_{n=1}^\infty a_F(n)n^{-s}=
\exp\biggl(\sum_{n=2}^\infty\frac{f(n)}{\log{n}}n^{\frac12-s}\biggr)
\quad\text{for }\Re(s)>1;
$$
we call $d_F=2\lim_{x\to0^+}xK(x)$ the \emph{degree} of $F$
and $Q_F=e^{-2f(1)}$ its \emph{analytic conductor};
and we say that $F$ is \emph{positive} if there are at most finitely
many $z\in\C$ with $m(z)<0$.
\end{definition}

Let $\LL$ denote the set of all $L$-data and $\LL^+\subseteq\LL$
the subset of positive elements.  Note that $\LL$ is a group with
respect to addition, with identity element $(0,0,0)$, and $\LL^+$
is a monoid.  For any $d\in\R$, let $\LL_d=\{F\in\LL:d_F=d\}$
and $\LL_d^+=\LL_d\cap\LL^+$.
\begin{examples}\hspace{1mm}
\begin{enumerate}
\item
If $L(s)=\exp\bigl(\sum_{n=2}^\infty b(n)n^{-s}\bigr)$
is an element of the Selberg class with complete $L$-function
$$
\Phi(s)=\epsilon Q^s\prod_{j=1}^k\Gamma(\lambda_js+\mu_j)\cdot L(s),
$$
then there is an $L$-datum $F=(f,K,m)\in\LL^+$ satisfying
$d_F=2\sum_{j=1}^k\lambda_j$, $L_F(s)=L(s)$,
$$
f(n)=\begin{cases}
-\log{Q}-\Re\sum_{j=1}^k\lambda_j
\frac{\Gamma'}{\Gamma}(\frac{\lambda_j}2+\mu_j)&\text{if }n=1,\\
\frac{b(n)\log{n}}{\sqrt{n}}&\text{if }n>1,
\end{cases}
$$
$$
K(x)=\sum_{j=1}^k\frac{e^{-(\frac12+\frac{\mu_j}{\lambda_j})x}}
{1-e^{-\frac{x}{\lambda_j}}},
\quad\text{and}\quad
m(z)=\ord_{s=\frac12+iz}\Phi(s).
$$
Note in particular that the estimate
$\sum_{n\le x}|f(n)|^2\ll_\varepsilon x^\varepsilon$
follows from the Ramanujan hypothesis together with the
bound $b(n)\ll n^\theta$ (see \cite[Lemma in \S2]{murty}).
\item
If $\pi$ is a unitary cuspidal automorphic
representation of $\GL_d(\A_\Q)$ with conductor $q$,
$$
L(s,\pi_\infty)=\prod_{j=1}^d\Gamma_\R(s+\mu_j),
\quad
-\frac{L'}{L}(s,\pi)=\sum_{n=2}^\infty c_nn^{-s}
\quad\text{and}\quad
\Lambda(s,\pi)=L(s,\pi_\infty)L(s,\pi),
$$
then there is an $L$-datum $F=(f,K,m)\in\LL_d^+$ satisfying
$L_F(s)=L(s,\pi)$,
$$
f(n)=\begin{cases}
-\frac12\log{q}-
\Re\sum_{j=1}^d\frac{\Gamma_\R'}{\Gamma_\R}(\frac12+\mu_j)&\text{if }n=1,\\
\frac{c_n}{\sqrt{n}}&\text{if }n>1,
\end{cases}
$$
$$
K(x)=\sum_{j=1}^d\frac{e^{-(\frac12+\mu_j)x}}{1-e^{-2x}},
\quad\text{and}\quad
m(z)=\ord_{s=\frac12+iz}\Lambda(s,\pi).
$$
In this case, the estimate $\sum_{n\le x}|f(n)|^2\ll\log^2{x}$
for $x\ge2$ follows from the Rankin--Selberg method
(see \cite[(2.24)]{rs}), and the other conditions on $f$ and $K$
follow from partial results toward the Ramanujan conjecture \cite{lrs}.
\item If $\rho:\Gal(\overline{\Q}/\Q)\to\GL_d(\C)$ is an Artin
representation then there is an $L$-datum $F=(f,K,m)\in\LL_d$ with
$L_F(s)=L(s,\rho)$, and $f,K,m$ defined similarly to the case of
automorphic $L$-functions above. The Artin conjecture asserts
that $F$ is positive.
\end{enumerate}
\end{examples}

\begin{remarks}\hspace{1mm}
\begin{enumerate}
\item
Note that we do not require an Euler product, and in fact the primes make no
appearance in Definition~\ref{def:Ldatum}.
What effectively replaces this is the assumption of non-vanishing
outside the critical strip, which is implied by the absolute convergence
of $\log L_F(s)$ for $\Re(s)>1$.
In \cite{bt} it was shown in wide generality that this condition
essentially characterizes the Euler products among all Dirichlet series
associated to automorphic forms. For instance, it follows from
\cite[Theorem~1.1]{bt} and Theorem~\ref{thm:multone} below that
if $f\in S_k(\Gamma_1(N))$ is a classical holomorphic modular form then
there is an $L$-datum $F\in\LL_2^+$ with $L_F(s)=L(s+\frac{k-1}2,f)$
if and only if $f$ is a normalized newform and Hecke eigenform.
\item We have not imposed the Ramanujan bound $a_F(n)\ll_\varepsilon
n^\varepsilon$, largely to avoid excluding most of the automorphic
$L$-functions.  However, this has the side effect of
including some examples which might be deemed undesirable, e.g.\
$\zeta(2s-\frac12)$ is the $L$-function of some element of $\LL_2^+$.
We take the view that it is better to include a few
misfits in our definition than to throw out the baby with the bath
water, and one can always pass to a restricted subclass if this becomes
problematic.\footnote{There is also an argument if favor of keeping
examples like $\zeta(2s-\frac12)$ in the definition:
Shimura's integral representation for the symmetric square $L$-function
could be viewed as an extension of the Rankin--Selberg method to this
example, and that in turn was a key ingredient in the
proof of the Gelbart--Jacquet lift.}
\item Definition~\ref{def:Ldatum} is arguably simpler than
Definition~\ref{def:selbergclass}, since it makes no mention of
Euler products, the $\Gamma$-function or analytic continuation.
It is also more concrete, compared to the rather intangible notion of
analytic continuation, since one can study axiom (A4) as an
identity of unknowns to be solved for.  This was the essential point of
\cite[Proposition~4.2]{bhk}, which may be viewed as a prototype for our
Theorem~\ref{thm:converse} below.
\item The notion of analytic conductor as a measure of complexity of an
$L$-function was introduced in \cite{iwaniec-sarnak}.  Our formulation
is similar (but not identical) to that of \emph{log conductor}
in \cite{cfkrs}. We make no claims that this formulation is the most
suitable in all contexts, but it at least has the feature of being
canonically defined, as shown by Theorem~\ref{thm:multone} below.
\end{enumerate}
\end{remarks}

\subsection{Main results}
The map $F\mapsto L_F$ defines a homomorphism from $\LL$ to the
multiplicative group of non-vanishing holomorphic functions on
$\{s\in\C:\Re(s)>1\}$. Our first result shows that this map is
injective, i.e.\ each $L$-datum is determined by its $L$-function, in
the following strong sense.
\begin{theorem}[Multiplicity one]
\label{thm:multone}
For $F=(f,K,m)\in\LL$, the following are equivalent:
\begin{itemize}
\item[(i)]$F=(0,0,0)$;
\item[(ii)]$\sum_{n=2}^\infty\frac{|f(n)|}{\log{n}}<\infty$;
\item[(iii)]$\sum_{n=1}^\infty\frac{|a_F(n)|}{\sqrt{n}}<\infty$;
\item[(iv)]$L_F(s)$ is a ratio of Dirichlet polynomials;
\item[(v)]$\sum_{\substack{z\in\supp(m)\\|\Re(z)|\le T}}|m(z)|=o(T)$.
\end{itemize}
\end{theorem}
\noindent
Thus, although we have chosen to promote the three components $f$,
$K$ and $m$ of our definition equally, without loss of generality one
can focus only on the $L$-functions, as in the Selberg class.

Next, we show that the classification of the degree $d<\frac53$ elements
of the Selberg class, begun by Conrey--Ghosh \cite{cg} and continued
and refined by Kaczorowski--Perelli \cite{kp1,kp53} and Soundararajan
\cite{sound}, can be adapted to our setting.  (We speculate that
Kaczorowski and Perelli's very intricate extension \cite{kp2} to degree
$<2$ could be adapted as well, but have not attempted to do so.)
\begin{theorem}[Converse theorem]\label{thm:converse}
Let $F\in\LL_d^+$ for some $d<\frac53$.
Then either $d=0$ and $L_F(s)=1$, or
$d=1$ and there is a primitive Dirichlet character $\chi$ and
$t\in\R$ such that $L_F(s)=L(s+it,\chi)$.
\end{theorem}

\subsection{Applications}
We describe three applications of Theorem~\ref{thm:converse}. The first
two concern the zeros of automorphic $L$-functions.
\begin{corollary}\label{cor:oddzeros}
Let $\pi$ be a unitary cuspidal automorphic representation of
$\GL_3(\A_\Q)$. Then its complete $L$-function $\Lambda(s,\pi)$
has infinitely many zeros of odd order.
\end{corollary}
\begin{proof}
Let $F=(f,K,m)\in\LL^+_3$ be the $L$-datum associated to $\pi$.
If $\Lambda(s,\pi)$ has at most finitely many zeros of odd order
then $m(z)$ is an even integer for all but at most finitely many $z$,
and thus $\frac12F\in\LL^+_{3/2}$, in
contradiction to Theorem~\ref{thm:converse}.
\end{proof}

\begin{corollary}\label{cor:distinctzeros}
For $j=1,2$, let $\pi_j$ be a unitary cuspidal automorphic representation
of $\GL_{d_j}(\A_\Q)$ with complete $L$-function $\Lambda(s,\pi_j)$.
If $d_2-d_1\le1$ and $\pi_1\not\cong\pi_2$ then
$\Lambda(s,\pi_2)/\Lambda(s,\pi_1)$ has infinitely many poles.
\end{corollary}
\begin{proof}
If $d_2<d_1$ then the conclusion follows by counting zeros, so we
may assume that $d_2\in\{d_1,d_1+1\}$.
Let $F\in\LL$ be the $L$-datum with $L$-function
$L_F(s)=L(s,\pi_2)/L(s,\pi_1)$, so that $d_F\in\{0,1\}$.
If $\Lambda(s,\pi_2)/\Lambda(s,\pi_1)$
has at most finitely many poles then $F$ is positive, so by
Theorem~\ref{thm:converse}, either $L_F(s)=1$ or
$L_F(s)=L(s+it,\chi)$ for some primitive Dirichlet character $\chi$ and
$t\in\R$. However, neither of these is possible since
$\pi_1\not\cong\pi_2$ and $\pi_2$ is cuspidal.
\end{proof}
\begin{remarks}\hspace{1mm}
\begin{enumerate}
\item
The assumption of cuspidality is only for ease of presentation,
and one could formulate versions of both of the above results for
products of cuspidal $L$-functions.
\item If $\pi_1$ and $\pi_2$ are unitary cuspidal automorphic
representations of $\GL_{d_1}(\A_\Q)$ and $\GL_{d_2}(\A_\Q)$
with $\pi_1\not\cong\pi_2$, the Grand Simplicity Hypothesis
predicts that $\Lambda(s,\pi_1)\Lambda(s,\pi_2)$ has at most finitely
many non-simple zeros.
Corollaries~\ref{cor:oddzeros} and \ref{cor:distinctzeros} give some
modest evidence in that direction. The fact that these results are new
is testimony of the difficulty of proving anything about the zeros of
high degree $L$-functions!
\item
Corollary~\ref{cor:distinctzeros} could likely be strengthened to
$d_2-d_1\le2$ by combining the methods of this paper with those of
\cite{bk}. Some special cases along these lines were demonstrated
by Raghunathan \cite{ra}.
\end{enumerate}
\end{remarks}

Our third application generalizes a result of Lemke-Oliver \cite{lo},
who considered totally multiplicative functions
$f:\Z_{>0}\to D=\{z\in\C:|z|\le1\}$ whose summatory functions
exhibit better than square-root cancellation relative to their mean-square
size, i.e.\
\begin{equation}\label{eq:cancellation}
\sum_{n\le x}|f(n)|^2\gg x\quad\text{and}\quad
\sum_{n\le x}f(n)\ll x^{\frac12-\delta}
\;\text{ for some }\delta>0.
\end{equation}
Lemke-Oliver noted that this holds if $f$ is a non-trivial
Dirichlet character, and asked if that is essentially the
only example. Although the problem appears to be intractable in
full generality, he was able to make progress for the subclass of
$f$ that are \emph{dictated by Artin symbols}, in the sense that there
is a Galois extension $K/\Q$ such that for every prime $p$ that does not
ramify in $K$, $f(p)$ depends only on the Frobenius conjugacy class
$\Frob_p$ at $p$.  His proof shows that for such $f$
there is a decomposition
$$
f(p)=\sum_\chi a_\chi\chi(\Frob_p)
$$
for all unramified primes $p$, where $\chi$ ranges over the characters
of the irreducible representations of $\Gal(K/\Q)$, and $a_\chi\in\Q$.
Thus, the Dirichlet series $\sum_{n=1}^\infty f(n)n^{-s}$ behaves like
an Artin $L$-function of degree $d=\sum_\chi a_\chi$, which is the value
of $f$ at split primes.

Lemke-Oliver concluded that $f$ must agree with a Dirichlet character
for almost all $p$ under the assumption that $d=1$, by
adapting Soundararajan's proof \cite{sound} of the classification
of degree $1$ elements of the Selberg class.  Note that $f(p)$
assumes only finitely many values for unramified $p$, each occurring
with positive density, so by the Selberg--Delange method, we expect
that the lower bound in \eqref{eq:cancellation} is only possible if
$f(p)\in\partial{D}=\{z\in\C:|z|=1\}$.  If that is indeed the case then we
must have $d=1$, but rather than attempting to justify that
heuristic, it suffices to appeal to Theorem~\ref{thm:converse}; in fact,
we obtain following stronger result.
\begin{corollary}
Let $f:\Z_{>0}\to\C$ be a totally multiplicative function dictated
by Artin symbols, with $|f(p)|<\frac53$ for all primes $p$.  If $f$
satisfies \eqref{eq:cancellation} then there is a primitive Dirichlet
character $\chi$ such that $f(p)=\chi(p)$ for all but at most finitely
many primes $p$.
\end{corollary}

\subsection{Concluding remarks}
Above we described three applications of our expanded notion of
$L$-functions.  However, our results so far have relied on essentially
the same arguments as those already applied to the Selberg class,
and it is natural to wonder whether this train of thought might also
lead to insights that go beyond those arguments, perhaps as far as a
classification of $\LL_d^+$ for some $d\ge2$.  In particular, can we
improve on the converse theorem for classical modular forms?

While we are unable to give any definitive answers to this question,
we offer a few philosophical remarks and suggestions for future work.
\begin{remarks}\label{rem:philosophy}
Let $\LL^{\mathrm{aut}}$ be the subgroup of $\LL$ generated by
the $L$-data associated to unitary cuspidal automorphic representations
of $\GL_d(\A_\Q)$ for all $d$. Presumably $\LL^{\mathrm{aut}}$ coincides with
the subset of $F\in\LL$ satisfying the Ramanujan bound
$a_F(n)\ll_\varepsilon n^\varepsilon$, but at present
we cannot prove an inclusion in either direction.
The elephant in the room is that $\LL^{\mathrm{aut}}$ is not only
a group, but has the additional structure of a commutative ring, at
least conjecturally.  Precisely, if $F_1,F_2\in\LL^{\mathrm{aut}}$ are
generators corresponding to cuspidal representations $\pi_1,\pi_2$,
then the Langlands functoriality conjecture predicts that there
is an automorphic representation with $L$-function equal to the
Rankin--Selberg product $L(s,\pi_1\times\pi_2)$, and we define
the product $F_1F_2\in\LL^{\mathrm{aut}}$ to be the $L$-datum with
that $L$-function.

The approach to classifying the elements of the Selberg class
taken so far purposefully ignores most of this structure and
relies essentially on Fourier analysis, which
amounts to considering twists by $n^{-it}$, i.e.\ multiplication (in the
above sense) by
$F\in\LL$ with $L_F(s)=\zeta(s+it)$. Note that such $F$ are units in
$\LL^{\mathrm{aut}}$, as are the $L$-data corresponding to $L(s+it,\chi)$
for primitive Dirichlet characters $\chi$.

Put in these terms, one cannot
help but wonder whether it would be more natural to build stability
under twist into the definition, at least by all of the units, i.e.\
to consider the subclass of $F\in\LL$ which have a twist $F_\chi\in\LL$
for every primitive character $\chi$.  For this subclass, it seems
likely that one could adapt the existing converse theorems for classical
holomorphic and Maass modular forms to classify the positive elements
of degree $2$.  (In fact, it might only be necessary to assume that $F$
is positive, and not all of the twists $F_\chi$, by following the method
of \cite{bk}.) Moreover, Cogdell and Piatetski-Shapiro conjectured
\cite[p.~166]{cps} that the analytic properties of twists by characters
should in general suffice to characterize the automorphic representations
among all irreducible admissible representations, so there is at least
some hope of eventually classifying everything of integral degree
this way.

However, there are a few subtleties that need to be considered before this
can be carried out.  First, in all known versions of the converse theorem
for degree at least $2$ (beginning with Weil \cite{weil2}), knowledge
of the relationship between the root numbers and conductors of a given
$L$-series and its twists is essential in the proof.  On the other hand,
Definition~\ref{def:Ldatum} does not even mention the root number (it
makes only a brief appearance in the proof of Proposition~\ref{prop:acfe},
as a constant of integration), and as our results demonstrate, it plays
no role in the classification of low-degree elements of $\LL^+$.

Second, the role of the Euler product in the converse theorem is similarly
hazy.  It has been conjectured that the degree $2$ $L$-functions with
Euler products can be characterized by a converse theorem without any
twists, but this is known to be false if one drops the Euler product
assumption. (In the other direction, with the added information from
twists, Weil's converse theorem does not require an Euler product.)
More generally, there are examples of Dirichlet series (e.g.\ certain
Shintani zeta-functions \cite{thorne}) which, together with all of their
character twists, have meromorphic continuation and satisfy a functional
equation, but are not associated to automorphic representations. These
examples do not contradict Cogdell and Piatetski-Shapiro's conjecture
since they lack Euler products.  Thus, the Euler product seems to be an
important hypothesis for characterizing automorphic representations with
minimal analytic data, but it is far from clear why this is so.

In our definition we offered a weaker alternative (non-vanishing outside
the critical strip, implied by axiom (A1))
as a possible substitute, and we speculate that it
may help to shed light on the matter.  In any case, we find it likely
that in order to make progress on the classification for degree $2$ and
beyond, one must first clarify the roles that the Euler product and
root number play in the converse theorem.  As tentative steps in this
direction, we issue the following challenges:
\begin{enumerate}
\item Prove a converse theorem for classical holomorphic modular forms,
assuming that all character twists satisfy the expected analytic
properties, but \emph{without knowledge of the root number}.
\item Prove a converse theorem for automorphic representations of
$\GL_3(\A_\Q)$, assuming axiom (A1) and that all character twists have
the expected analytic properties, but \emph{without requiring an Euler
product}.
\end{enumerate}
Of course it might be that one or both of these is impossible,
in which case a proof that there is no such result would be even more
interesting!
\end{remarks}

\subsection*{Acknowledgements}
This work was carried out during a year-long stay at the Research
Institute for Mathematical Sciences, Kyoto, Japan. It is a pleasure
to thank all of the RIMS staff, in particular my host, Akio Tamagawa,
for their generous hospitality. I would also like to thank Akihiko Yukie
for organizing the Conference on Automorphic Forms at Kyoto University
in June 2013, which provided the impetus for this work.  Finally,
I thank Frank Thorne for performing some computations in relation to
Remarks~\ref{rem:philosophy}, and Brian Conrey, David Farmer, Peter Sarnak
and Akshay Venkatesh for helpful suggestions.

\section{Basic properties}
In this section we establish the basic properties of the $L$-functions
associated to elements of $\LL$, culminating in the proof of
Theorem~\ref{thm:multone}.  First, we show that the derivation of the
explicit formula for $L$-functions can be inverted to prove that for
any $F\in\LL$, a suitably ``completed'' form of $L_F(s)$ has meromorphic
continuation and satisifies a functional equation. In particular,
we construct a canonical notion of ``$\Gamma$-factor'' associated to $F$,
as follows.
\begin{proposition}[Meromorphic continuation and functional equation]
\label{prop:acfe}
For every $F=(f,K,m)\in\LL$,
there is a function $\gamma_F(s)$, defined uniquely up
to scaling by elements of $\R^\times$,
with the following properties:
\begin{enumerate}
\item[(i)] $\log\gamma_F(s)$ is holomorphic for $\Re(s)>\frac12$, 
and $\frac{d^n}{ds^n}\log\gamma_F(s)$ extends continuously to 
$\Re(s)\ge\frac12$ for each $n\ge0$;
\item[(ii)]
there are constants $d,c_{-1}\in\R$ and $\mu,c_0,c_1,\ldots\in\C$ such that
$$
\log\gamma_F(s)=
\left(s-\frac12\right)\left(\frac{d}2\log\frac{s}{e}+c_{-1}\right)
+\frac{\mu}2\log\frac{s}{e}
+\sum_{j=0}^{n-1}\frac{c_j}{s^j}+O_n(|s|^{-n}),
$$
uniformly for $\Re(s)\ge\frac12$ and any fixed $n\ge0$;
\item[(iii)] the product $\Lambda_F(s)=\gamma_F(s)L_F(s)$
continues meromorphically to
$$
\Omega=\C\setminus
\bigcup_{\{z\in\C:m(z)\notin\Z\}}\Bigl[
\bigl(\tfrac12+i(-\infty,\Re(z)]\bigr)\cup
\bigl(\bigl[\tfrac12-|\Im(z)|,\tfrac12+|\Im(z)|\bigr]+i\Re(z)\bigr)
\Bigr]
$$
and has meromorphic finite order, i.e.\
$\Lambda_F(s)=h_1(s)/h_2(s)$, where $h_1$ and $h_2$ are holomorphic on
$\Omega$, and there is a number $A\ge0$ such that
for any closed subset $E\subseteq\Omega$ we have
$h_1(s),h_2(s)\ll_E\exp(|s|^A)$ for all $s\in E$;
\item[(iv)] the functional equation
$\Lambda_F(s)=\overline{\Lambda_F(1-\bar{s})}$ holds as an identity of
meromorphic functions on $\Omega$;
\item[(v)]
$\frac{\Lambda_F'}{\Lambda_F}(s)$ continues meromorphically to
$\C$, with at most simple poles, and satisfies
$$
\Res_{s=\frac12+iz}\frac{\Lambda_F'}{\Lambda_F}(s)=m(z)
\quad\text{for all }z\in\C.
$$
In particular, $\supp(m)\subseteq\{z\in\C:|\Im(z)|\le\frac12\}$.
\end{enumerate}
\end{proposition}

\begin{remark}
The proof of Proposition~\ref{prop:acfe} shows that the number
$d$ appearing in (ii) is the degree $d_F=2\lim_{x\to0^+}xK(x)$,
and $\mu=-2\lim_{x\to0^+}\frac{d}{dx}(xK(x))$.
\end{remark}

\subsection{Lemmas}
We begin with a few lemmas, the first of which establishes the basic
equivalence between distributional identities and functions possessing
analytic continuation and a functional equation.  In what follows we
denote by $H$ the set of entire functions $h$ such that
$h(\R)\subseteq\R$ and the Fourier transform
$g(x)=\frac1{2\pi}\int_\R h(t)e^{-ixt}\,dt$ is smooth of compact support.
\begin{lemma}\label{lem:ft}
Let $\varphi$ be a holomorphic function on
$\{z\in\C:\Im(z)<y\}$ for some $y\in\R$, and suppose that
$\varphi(z)$ has at most polynomial growth as $|z|\to\infty$ in any fixed
horizontal strip $\{z\in\C:\Im(z)\in[a,b]\}$ with $a<b<y$.
Fix $c<y$, and suppose that $\int_{\Im(z)=c}\varphi(z)h(z)\,dz\in\R$
for every $h\in H$. Then $\varphi$ continues to an entire function,
with at most polynomial growth in horizontal strips, and satisfies the
functional equation $\varphi(\bar{z})=\overline{\varphi(z)}$.
\end{lemma}
\begin{proof}
Consider the integral
$$
u(x)=\frac1{2\pi}\int_{\Im(z)=c}\varphi(z)e^{-z^2-ixz}\,dz.
$$
Since $\varphi(z)e^{-z^2}$ is holomorphic and of rapid decay in
horizontal strips
for $\Im(z)<y$, $u(x)$ is independent of the value of $c$.
For any fixed $x\in\R$, $e^{-z^2}\cos(xz)$ and
$e^{-z^2}\sin(xz)$ are real valued for $z\in\R$, so it follows by a
standard approximation argument that
$$u(x)+u(-x)=\frac1{\pi}\int_{\Im(z)=c}\varphi(z)e^{-z^2}\cos(xz)\,dz$$
and
$$i(u(x)-u(-x))=\frac1{\pi}\int_{\Im(z)=c}\varphi(z)e^{-z^2}\sin(xz)\,dz$$
are real valued, so that $u(-x)=\overline{u(x)}$.
Combining this with the trivial estimate $u(x)\ll_c e^{cx}$,
we get $u(x)\ll_c e^{c|x|}$ for all $c<y$.

Together with the Fourier inversion formula
$$\varphi(z)e^{-z^2}=\int_{-\infty}^{\infty}u(x)e^{ixz}\,dx,$$
this shows that $\varphi(z)$ continues to an entire function,
has finite order in any fixed horizontal strip, and satisfies
$\varphi(\bar{z})=\overline{\varphi(z)}$.  Finally, the
Phragm\'en--Lindel\"of convexity principle applied to $\varphi$ on the strip
$\{z\in\C:|\Im(z)|\le1+|y|\}$ shows that $\varphi$ has at most polynomial
growth in horizontal strips.
\end{proof}

\begin{lemma}\label{lem:msymmetry}
Let $(f,K,m)\in\LL$.  Then $m(\bar{z})=m(z)$ for all $z\in\C$.
\end{lemma}
\begin{proof}
Put $m'(z)=m(z)-m(\bar{z})$. Clearly $m'(z)=0$ for all $z\in\R$, and
we aim to show that this holds for all $z\in\C$.
By axiom (A3), there is a positive integer $M$ such that
$\sum_{z\in\supp(m')}\left|\frac{m'(z)}{z^M}\right|<\infty$.
With this choice of $M$, let
$$
q(z)=\sum_{z_0\in\supp(m')}\frac{m'(z_0)}{z-z_0}
\left(\frac{z}{z_0}\right)^{M-1}.
$$
Then $q$ is meromorphic on $\C$ with at most simple poles,
satisfies $\Res_{z=z_0}q(z)=m'(z_0)$ for all $z_0\in\C$,
and $q(\bar{z})=-\overline{q(z)}$.  Further, setting
$y=\sup\{|\Im(z)|:z\in\supp(m')\}$, $q$ is holomorphic for $\Im(z)<-y$
and has at most polynomial growth in any strip $\{z\in\C:\Im(z)\in[a,b]\}$
with $a<b<-y$.

Next, let $h\in H$.  By axiom (A4), we have
$\sum_{z\in\C}m(z)h(z)\in\R$, and hence
$$
\sum_{z\in\C}m(z)h(z)=\overline{\sum_{z\in\C}m(z)h(z)}
=\sum_{z\in\C}m(z)h(\bar{z})
=\sum_{z\in\C}m(\bar{z})h(z).
$$
Thus, for any $c>y$,
\begin{align*}
0&=\sum_{z\in\C}m'(z)h(z)
=\frac1{2\pi i}\int_{\Im(z)=-c}q(z)h(z)\,dz
-\frac1{2\pi i}\int_{\Im(z)=c}q(z)h(z)\,dz\\
&=\frac1{\pi i}\Re\int_{\Im(z)=-c}q(z)h(z)\,dz,
\end{align*}
i.e.\ $\int_{\Im(z)=-c}q(z)h(z)\,dz\in i\R$.
Hence, by Lemma~\ref{lem:ft} with $\varphi(z)=iq(z)$, $q$ is entire,
and $m'(z)=0$ identically.
\end{proof}

\subsection{Proof of Proposition~\ref{prop:acfe}}
Set $\mu=-2\lim_{x\to0^+}\frac{d}{dx}(xK(x))$ and
$K_1(x)=K(x)-\frac{d_F}{4\sinh(x/2)}+\frac{\mu}{2\cosh(x/2)}$.
Then $K_1(x)/x$ extends to a Schwartz function on $\R$, so
\begin{equation}\label{eq:kdef}
k(z)=i\int_0^\infty\frac{K_1(x)}{x}e^{-ixz}\,dx
\end{equation}
is well defined for $\Im(z)\le0$ and holomorphic for $\Im(z)<0$.
For any $n\ge0$ we have
$$
k^{(n)}(z)=\int_0^\infty K_1(x)(-ix)^{n-1}e^{-ixz}\,dx
$$
for $\Im(z)<0$, and by a standard argument based on Lebesgue's dominated
convergence theorem, this extends continuously to $\Im(z)\le0$.

Next let $h\in H$ with Fourier transform $g$.
By Plancherel's theorem, we have
$$\int_0^\infty K_1(x)g(x)\,dx=\frac1{2\pi}\int_\R k'(t)h(t)\,dt
=\frac1{2\pi}\int_{\Im(z)=-c}k'(z)h(z)\,dz$$
for any $c\ge0$.  Together with the identities
$$
\int_0^\infty\frac{g(0)-g(x)}{2\sinh(x/2)}\,dx
=g(0)\log(4e^\gamma)+\frac1{2\pi}\int_\R
\frac{\Gamma'}{\Gamma}\!\left(\frac12+it\right)h(t)\,dt
$$
and
$$
\int_0^\infty\frac{g(0)-g(x)}{2\cosh(x/2)}\,dx
=g(0)\frac{\pi}2+\frac1{2\pi}\int_\R\frac12\left[
\frac{\Gamma'}{\Gamma}\!\left(\frac14+\frac{it}2\right)
-\frac{\Gamma'}{\Gamma}\!\left(\frac34+\frac{it}2\right)
\right]h(t)\,dt,
$$
this yields
\begin{align*}
\Re\int_0^\infty K(x)(g(0)-g(x))\,dx
=\frac1{2\pi}\Re\int_{\Im(z)=-c}\varphi(z)h(z)\,dz,
\end{align*}
where
$$
\varphi(z)=
\frac{d_F}{2}\frac{\Gamma'}{\Gamma}\!\left(\frac12+iz\right)
+\frac{\mu}2\left[\frac{\Gamma'}{\Gamma}\!\left(\frac34+\frac{iz}2\right)
-\frac{\Gamma'}{\Gamma}\!\left(\frac14+\frac{iz}2\right)\right]
-k'(z)+C
$$
and
$C=\frac{d_F}{2}\log(4e^\gamma)+\Re\bigl(k'(0)-\frac{\pi\mu}2\bigr)$.

Let $\omega\in\C^\times$ be a constant of modulus $1$, to be
determined below, and set
\begin{equation}\label{eq:gammaFdef}
\gamma_F(s)=\omega\Gamma(s)^{\frac{d_F}2}
\left(\frac{\Gamma((s+1)/2)}{\Gamma(s/2)}\right)^{\mu}
\exp\!\bigl[(C-f(1))(s-\tfrac12)-ik\bigl(-i(s-\tfrac12)\bigr)\bigr],
\end{equation}
where for any $z\in\C$ we define
$\Gamma(s)^z=\exp(z\log\Gamma(s))$ for $\Re(s)>0$ using the principal
branch of $\log\Gamma$. Note that
$\frac{\gamma_F'}{\gamma_F}(\frac12+iz)=\varphi(z)-f(1)$,
and by the above we see that
$\log\gamma_F(s)$ and all of its derivatives are holomorphic for
$\Re(s)>\frac12$ and extend continuously to $\Re(s)\ge\frac12$, which
establishes (i).

To see (ii), we rewrite \eqref{eq:kdef} in the form
$$
k(z)=i\int_0^\infty\frac{K_1(x)}{x}e^{x/2}e^{-(\frac12+iz)x}\,dx
$$
and apply integration by parts repeatedly to see that
$$
k(z)=\sum_{j=1}^{n-1}\frac{c_j}{(\frac12+iz)^j}
+O_n\bigl(|\tfrac12+iz|^{-n}\bigr)
$$
for some constants $c_j$. Using this together with Stirling's formula in 
\eqref{eq:gammaFdef} yields (ii).

Next let $\Lambda_F(s)=\gamma_F(s)L_F(s)$ and
$\Phi(z)=\Lambda_F(\frac12+iz)$. Then $\Phi(z)$ is
analytic for $\Im(z)<-\frac12$, where it satisfies
$$
-i\frac{\Phi'}{\Phi}(z)=\varphi(z)-\sum_{n=1}^{\infty}\frac{f(n)}{n^{iz}}.
$$
Thus, for any $c>\frac12$ we have
\begin{equation}\label{eq:Phiint}
\begin{aligned}
\frac1{\pi}\Im\int_{\Im(z)=-c}\frac{\Phi'}{\Phi}(z)h(z)\,dz
&=2\Re\left[\frac1{2\pi}\int_{\Im(z)=-c}\varphi(z)h(z)\,dz
-\sum_{n=1}^{\infty}\frac{f(n)}{2\pi}\int_{\Im(z)=-c}h(z)n^{-iz}\,dz\right]\\
&=2\Re\left[\int_0^\infty K(x)(g(0)-g(x))\,dx
-\sum_{n=1}^{\infty}f(n)g(\log{n})\right].
\end{aligned}
\end{equation}

As in the proof of the Lemma~\ref{lem:msymmetry}, we define
$$
q(z)=\frac{m(0)}{z}+\sum_{z_0\in\supp(m)\setminus\{0\}}
\frac{m(z_0)}{z-z_0}
\left(\frac{z}{z_0}\right)^{M-1}
$$
for a suitable positive integer $M$. Lemma~\ref{lem:msymmetry} shows that
$m(\bar{z})=m(z)$, so that $q(\bar{z})=\overline{q(z)}$.
Let $y=\sup\{|\Im(z)|:z\in\supp(m)\}$ and $h\in H$ with Fourier
transform $g$. Then for any $c>\max(y,\frac12)$, we have
\begin{align*}
\sum_{z\in\C}m(z)h(z)
&=\frac1{2\pi i}\int_{\Im(z)=-c}q(z)h(z)\,dz
-\frac1{2\pi i}\int_{\Im(z)=c}q(z)h(z)\,dz\\
&=\frac1{\pi}\Im\int_{\Im(z)=-c}q(z)h(z)\,dz.
\end{align*}
On the other hand, by axiom (A4) and \eqref{eq:Phiint},
this equals
$$
2\Re\left[\int_0^\infty K(x)(g(0)-g(x))\,dx
-\sum_{n=1}^\infty f(n)g(\log{n})\right]=
\frac1{\pi}\Im\int_{\Im(z)=-c}\frac{\Phi'}{\Phi}(z)h(z)\,dz.
$$
Thus, $\int_{\Im(z)=-c}\left(\frac{\Phi'}{\Phi}(z)-q(z)\right)h(z)\,dz$
is real valued, so by Lemma~\ref{lem:ft} it follows that
$\frac{\Phi'}{\Phi}(z)-q(z)$ continues to an entire function
with at most polynomial growth in horizontal strips, and we
have the functional equation
$\frac{\Phi'}{\Phi}(\bar{z})=\overline{\frac{\Phi'}{\Phi}(z)}$.
This yields the residue formula
$$
m(z_0)=\Res_{z=z_0}q(z)=\Res_{z=z_0}\frac{\Phi'}{\Phi}(z)
=\Res_{s=\frac12+iz_0}\frac{\Lambda_F'}{\Lambda_F}(s)
$$
and establishes (v).

Define
$$
l(z)=
\begin{cases}
\log(z-i)+\log(z+i)&\text{if }\Re(z)>0\text{ or }|\Im(z)|>1,\\
\log(z^2+1)+2\pi i&\text{if }\Re(z)<0\text{ and }\Im(z)>0,\\
\log(z^2+1)-2\pi i&\text{if }\Re(z)<0\text{ and }\Im(z)<0,
\end{cases}
$$
where each $\log$ refers to the principal branch.
One can check that the definitions agree where they overlap, so
$l$ is analytic on $\C\setminus\bigl((-\infty,0]\cup i[-1,1]\bigr)$
and satisfies $\exp(l(z))=z^2+1$.
For any $z_0\in\C$ we set
$$
l_{z_0}(z)=
\begin{cases}
2\log(z-z_0)&\text{if }z_0\in\R,\\
2\log|\Im(z_0)|+l\!\left(\frac{z-\Re(z_0)}{|\Im(z_0)|}\right)
&\text{if }z_0\notin\R,
\end{cases}
$$
so that $l_{z_0}$ is analytic for
$z-\Re(z_0)\in\C\setminus\bigl((-\infty,0]\cup i[-|\Im(z_0)|,|\Im(z_0)|]\bigr)$
and satisfies $\exp(l_{z_0}(z))=(z-z_0)(z-\overline{z_0})$.

Next let $\Omega$ be as in the statement of the proposition and define
\begin{align*}
Q(z)=z^{m(0)}
\prod_{z_0\in\supp(m)\setminus\{0\}}
\exp\left(m(z_0)\sum_{n=1}^{M-1}\frac{(z/z_0)^n}n\right)
\begin{cases}
\exp\bigl(\tfrac12m(z_0)l_{z_0}(z)\bigr)&\text{if }m(z_0)\notin\Z,\\
(z-z_0)^{m(z_0)}&\text{if }m(z_0)\in\Z,
\end{cases}
\end{align*}
where $z^{m(0)}$ means $\exp(m(0)\log{z})$ if $m(0)\notin\Z$.
Since $m(\bar{z})=m(z)$ for all $z$, we see that $Q$ is meromorphic on
$\{z\in\C:\frac12+iz\in\Omega\}$ and satisfies $\frac{Q'}{Q}(z)=q(z)$
in that region. By the above we conclude that $\Phi(z)/Q(z)$ continues
to an entire, non-vanishing function of finite order. It follows that
$\Lambda_F(s)$ continues meromorphically to $\Omega$ and has meromorphic
finite order, which establishes (iii).

Since $\Omega$ is simply connected, by
integrating the functional equation for
$\frac{\Lambda_F'}{\Lambda_F}(s)$ we get
$\Lambda_F(s)=c\omega^2\overline{\Lambda_F(1-\bar{s})}$ for some constant
$c\in\C^\times$. Consideration of this equation for $\Re(s)=\frac12$
shows that $|c|=1$, and we choose $\omega$ to satisfy $c\omega^2=1$.
This establishes (iv).

It remains only to see that $\gamma_F$ is unique up to multiplication by
a non-zero real scalar.
To that end, suppose that $\tilde\gamma_F(s)$
is another function with the same properties, and consider the ratio
$r(s)=\frac{\tilde\gamma_F(s)}{\gamma_F(s)}$.
Since we may also write $r(s)$ in the form
$\frac{\tilde\gamma_F(s)L_F(s)}{\gamma_F(s)L_F(s)}$, it follows from
(iv) and (v) that
$\frac{r'}{r}(s)$ has meromorphic continuation to $\C$ and satisfies
$\frac{r'}{r}(s)=-\overline{\frac{r'}{r}(1-\bar{s})}$. By
(i), $\frac{r'}{r}(s)$ is continuous on $\Re(s)\ge\frac12$,
so it must be entire. Therefore, $r(s)$ is entire and non-vanishing, and
satisfies $r(s)=\overline{r(1-\bar{s})}$.
Further, by (ii), there are numbers $d',c'_{-1}\in\R$
and $\mu',c'_0\in\C$ such that
$r(s)=\exp\bigl((\frac{d'}2\log\frac{s}e+c'_{-1})(s-\frac12)
+\frac{\mu'}2\log\frac{s}e+c'_0\bigr)\bigl(1+O(|s|^{-1})\bigr)$
for $\Re(s)\ge\frac12$.
The functional equation implies that $r(\frac12+it)\in\R$ for all
$t\in\R$, and taking $t\to\infty$ we conclude that
$d'=c'_{-1}=\Im(\mu')=0$.
Together with the functional equation, this implies that
$r(s)\ll(1+|s-\frac12|)^{\mu'}$, and thus $r$ is a polynomial.
Since it does not vanish, it must be a non-zero constant, and invoking
the functional equation once more, it is real valued.

\subsection{Proof of Theorem~\ref{thm:multone}}
Clearly
$(\mathrm{i})\Longrightarrow(\mathrm{ii})\Longrightarrow(\mathrm{iii})$,
and we will show that
$(\mathrm{iii})\Longrightarrow(\mathrm{iv})\Longrightarrow(\mathrm{v})
\Longrightarrow(\mathrm{i})$.
Beginning with (iii), 
suppose that $F=(f,K,m)\in\LL$ satisfies
$\sum_{n=1}^\infty\frac{|a_F(n)|}{\sqrt{n}}<\infty$.
Then $L_F(s)$ is holomorphic for $\Re(s)>\frac12$ and extends
continuously to $\Re(s)\ge\frac12$.
By Proposition~\ref{prop:acfe},
$\Lambda_F(\frac12+it)=\gamma_F(\frac12+it)L_F(\frac12+it)$ is
well defined and real valued for all sufficiently large $t>0$.
Thus, we have
$$
L_F(\tfrac12+it)=\overline{L_F(\tfrac12+it)}
e^{-2i\arg\gamma_F(\frac12+it)},
$$
where $\arg\gamma_F(\frac12+it)$ denotes the continuous
extension of $\Im\log\gamma_F(s)$ to $\Re(s)=\frac12$.

For a fixed positive integer $m$ we multiply both sides
of this equation by
$m^{it}$ and take the average over $t\in[T,2T]$ as $T\to\infty$.
On the left-hand side, thanks to the absolute convergence of
$\sum a_F(n)/\sqrt{n}$,
this is $a_F(m)/\sqrt{m}+o(1)$, while
on the right-hand side we get
$$
\frac1T\int_T^{2T}\sum_{n=1}^\infty\frac{\overline{a_F(n)}}{\sqrt{n}}
e^{it\log(mn)-2i\arg\gamma_F(\frac12+it)}\,dt
=\frac1T\int_T^{2T}\sum_{n=1}^\infty\frac{\overline{a_F(n)}}{\sqrt{n}}
e^{i(t\log(mn)-\varphi(t)+O(1/t))}\,dt,
$$
where, by Proposition~\ref{prop:acfe}(ii),
$\varphi(t)=d_Ft\log(t/e)+2c_{-1}t+\Im(\mu)\log(t/e)+\theta$,
for some constants $c_{-1},\theta\in\R$ and $\mu\in\C$.
We may ignore the $O(1/t)$ error term at the cost of $o(1)$.
Also, since the sum over $n$ is absolutely convergent, 
the terms for $n>\log{T}$ contribute only $o(1)$, so we have
$$
\frac{a_F(m)}{\sqrt{m}}=o(1)+
\frac1T\int_T^{2T}\sum_{n\le\log{T}}\frac{\overline{a_F(n)}}{\sqrt{n}}
e^{i(t\log(mn)-\varphi(t))}\,dt.
$$

If $d_F\ne0$ then it is easy to see that $t\log(mn)-\varphi(t)$ has no
stationary points in $[T,2T]$ for sufficiently large $T$, and an
application of integration by parts shows that the right-hand side is
$o(1)$. To avoid a contradiction for $m=1$, we must have $d_F=0$.
The main term of $t\log(mn)-\varphi(t)$ is thus $(\log(mn)-2c_{-1})t$,
so there is still no stationary phase unless $mn=e^{2c_{-1}}$. In
particular, $e^{2c_{-1}}$ must be a positive integer, say $q$, and we
have $a_F(m)=0$ unless $m|q$. Thus, $L_F(s)$ is a Dirichlet polynomial,
which clearly implies (iv).

Next, assuming (iv), $a_F(n)$ is supported on the $y$-smooth numbers
$$\{n\in\Z_{>0}:p\text{ prime and }p|n\Longrightarrow p\le y\}$$
for some $y>0$. Taking the logarithm, the same conclusion applies to
$f(n)$.
By axiom (A1), we have the estimate $f(n)\log^k{n}\ll_k1$ for every $k$,
and applying this with $k=\pi(y)+1$, we see that
$\frac{L_F'}{L_F}(s)=-\sum_{n=2}^\infty f(n)n^{\frac12-s}$ is
holomorphic for $\Re(s)>\frac12$ and continuous on the boundary.
Hence this applies to $\frac{\Lambda_F'}{\Lambda_F}(s)=
\frac{\gamma_F'}{\gamma_F}(s)+\frac{L_F'}{L_F}(s)$ as well, and together
with the functional equation, this implies that $\Lambda_F(s)$ is entire
and non-vanishing. Therefore $m(z)=0$ identically, which implies (v).

Finally, let us assume (v), and set
$N(t)=\sum_{\substack{z\in\supp(m)\\|\Re(z)|\le{t}}}|m(z)|$ for $t\ge0$.
Then by hypothesis, there is a
function $\varepsilon(T)$ such that
$\lim_{T\to\infty}\varepsilon(T)=0$ and
$N(t)\le\varepsilon(T)t$ for all $t\ge T$.
We fix a test function $g_0$ which is
non-negative, even, smooth, supported on $[-1,1]$, satisfies $g_0(0)=1$,
and has Fourier transform $h_0\in H$.
For $\theta\in\R$, $T>0$ and $x_0>0$, we consider axiom (A4) applied to
$g(x)=e^{i\theta}g_0(T(x-x_0))+e^{-i\theta}g_0(T(x+x_0))$, with Fourier
transform $h(z)=2T^{-1}\cos(\theta+x_0z)h_0(T^{-1}z)$.
If we choose
$x_0=\log{n}$ for some integer $n\ge2$ and pick $\theta$ so that
$e^{i\theta}f(n)\in\R_{\ge0}$, then it is straightforward to
see that
\begin{equation}\label{eq:2fn}
2\Re\left[\int_0^\infty{K(x)}(g(0)-g(x))\,dx
-\sum_{n=1}^{\infty}f(n)g(\log n)\right]
=-2|f(n)|+o(1)
\end{equation}
as $T\to\infty$.

On the other hand, by axiom (A4) this equals
$$
\frac2T
\sum_{z\in\supp(m)}m(z)\cos(\theta+z\log{n})h_0(T^{-1}z).
$$
Since $g_0$ is smooth, $h_0$ decays rapidly in horizontal strips;
in particular, for $T\ge1$,
$$|\cos(\theta+z\log{n})h_0(T^{-1}z)|\le
C_n(1+|\Re(T^{-1}z)|)^{-2}\quad\text{for }|\Im(z)|\le\frac12$$
holds from some $C_n>0$ depending only on $n$.
Further, since $N(t)$ is non-decreasing, we have
$N(t)\le\varepsilon(T)\max(t,T)$ for all $t\ge0$. Thus, for $T\ge1$ we
get
\begin{align*}
\biggl|\frac2T&
\sum_{z\in\supp(m)}m(z)\cos(\theta+z\log{n})h_0(T^{-1}z)\biggr|
\le\frac{2C_n}T\int_0^\infty\left(1+\frac{t}{T}\right)^{-2}\,dN(t)\\
&=\frac{4C_n}{T^2}\int_0^\infty\left(1+\frac{t}{T}\right)^{-3}N(t)\,dt
\le\frac{4C_n\varepsilon(T)}{T^2}\int_0^\infty\left(1+\frac{t}{T}\right)^{-3}
\max(t,T)\,dt\\
&=3C_n\varepsilon(T)=o(1).
\end{align*}

Together with \eqref{eq:2fn}, this shows that $f(n)=0$ for all $n\ge2$.
In particular, $L_F(s)=1$, which implies (iv); in turn, as we saw above,
this implies that $m(z)=0$ identically.  Therefore,
$$
\Re\biggl[\int_0^\infty K(x)(g(0)-g(x))\,dx-f(1)g(0)\biggr]=0.
$$
for every suitable test function $g$. Choosing $g$ as above for an
arbitrary $x_0>0$ and $\theta$ so that $e^{i\theta}K(x_0)\in\R_{\ge0}$,
we see that $K(x_0)=0$. Finally, $f(1)\in\R$ by axiom (A1), so we have
$f(1)=0$.  Thus, (i) holds and this completes the proof.

\section{Proof of Theorem~\ref{thm:converse}}
Let $F\in\LL_d^+$ for some $d<2$.
The main object of study in the method initiated by Conrey and Ghosh
\cite{cg} is the exponential sum
$$
S_F(z)=\sum_{n=1}^\infty a_F(n)e(nz),
$$
defined for $z\in\{x+iy\in\C:y>0\}$.
For $k\in\Z$ we write
$$
S_F^{(k)}(z)=\sum_{n=1}^\infty a_F(n)(2\pi in)^ke(nz).
$$
Note that this is just the $k$th derivative for $k\ge0$.
Since $S_F(z)$ is periodic and decays exponentially as
$\Im(z)\to\infty$, it will be enough to consider $z$ in a box
$$
B=\{z=-x+iy\in\C:x\in[x_1,x_2], y\in(0,y_1]\},
$$
for fixed $x_2>x_1>0$ and $y_1>0$ to be specified later.

\begin{lemma}\label{lem:SFkint}
For $F$ as above, let $c_{-1}\in\R$ and $\mu\in\C$
be the constants given by Proposition~\ref{prop:acfe}, and define
$$
G(s)=\bigl(2\pi(1-\tfrac{d}2)^{1-\frac{d}2}e^{c_{-1}}\bigr)^{\frac12-s}
\Gamma\!\left(\left(1-\frac{d}2\right)\left(s-\frac12\right)
+\frac{1-\mu}2\right).
$$
Then
for any integer $k\ge0$ there are constants $c_{kj}\in\C$, $0\le
j\le k$, with $c_{kk}\ne0$, such that for any
$\sigma>\max\bigl(1,\frac12+\frac{\Re(\mu)-1}{2-d}\bigr)$
and $\varepsilon>0$,
$$
z^kS_F^{(k)}(z)=
O_{B,k,\varepsilon}(\Im(z)^{-\varepsilon})+\sum_{j=0}^k\frac{c_{kj}}{2\pi i}
\int_{\Re(s)=\sigma}\Lambda_F(s)
G\!\left(s+\frac{2j}{2-d}\right)(-iz)^{-s}\,ds,
$$
uniformly for $z\in B$.
\end{lemma}

\begin{proof}
By Proposition~\ref{prop:acfe}(ii) and Stirling's formula, we find that
there are constants $\alpha_0,\alpha_1,\ldots\in\C$ such that
\begin{align*}
\log\frac{\gamma_F(s)G(s)}{(2\pi)^{-s}\Gamma(s)}
=\sum_{j=0}^{n-1}\frac{\alpha_j}{s^j}+O_n\bigl(|s|^{-n}\bigr),
\end{align*}
uniformly on
$\{s\in\C:\Re(s)\ge\frac12\}\cap
\{s\in\C:\Re(s)\ge\frac12+\frac{\Re(\mu)}{2-d}
\text{ or }|\Im(s)-\frac{\Im(\mu)}{2-d}|\ge1\}$.
We take the exponential to get
$$
\gamma_F(s)G(s)
=(2\pi)^{-s}\Gamma(s)\left(\sum_{j=0}^{n-1}\frac{\beta_j}{s^j}
+O_n\bigl(|s|^{-n}\bigr)\right)
$$
for some constants $\beta_j\in\C$, with $\beta_0\ne0$.

Next we fix $k\ge0$, take $n=k+3$ in the above, and multiply by
$G(s+\frac{2k}{2-d})/G(s)$ to get
\begin{align*}
\gamma_F(s)G\!\left(s+\frac{2k}{2-d}\right)
&=(2\pi)^{-s}\Gamma(s+k)
\frac{G\bigl(s+\frac{2k}{2-d}\bigr)\Gamma(s)}
{G(s)\Gamma(s+k)}
\left(\sum_{j=0}^{k+2}\frac{\beta_j}{s^j}+O_k\bigl(|s|^{-k-3}\bigr)\right)\\
&=\sum_{j=-2}^k\gamma_{kj}
(2\pi)^{-s}\Gamma(s+j)+O_k\bigl(|s|^{-1})(2\pi)^{-s}\Gamma(s-2)
\end{align*}
for some $\gamma_{kj}\in\C$ with $\gamma_{kk}\ne0$,
uniformly on the set
\begin{equation}\label{eq:94contour}
\{s\in\C:\Re(s)\ge\tfrac94\}\cap
\{s\in\C:\Re(s)\ge\tfrac12+\tfrac{\Re(\mu)+2}{2-d}
\text{ or }|\Im(s)-\tfrac{\Im(\mu)}{2-d}|\ge1\}.
\end{equation}

Fix $\sigma>\max\bigl(1,\frac12+\frac{\Re(\mu)-1}{2-d}\bigr)$,
and consider the integral
$\frac1{2\pi i}\int_{\Re(s)=\sigma}\Lambda_F(s)
G(s+\frac{2k}{2-d})(-iz)^{-s}\,ds$.
Shifting the contour to the right if necessary, we may assume without
loss of generality that $\sigma\ge\frac94$.
We write
$M_k(s)=\gamma_F(s)^{-1}\sum_{j=-2}^k\gamma_{kj}(2\pi)^{-s}\Gamma(s+j)$,
$R_k(s)=G(s+\frac{2k}{2-d})-M_k(s)$, and split the integral
accordingly.
For the integral against $R_k(s)$,
we shift the contour to the boundary
of \eqref{eq:94contour}, on which we have the estimate
$\Lambda_F(s)R_k(s)(-iz)^{-s}\ll_{B,k}|s|^{-5/4}$
for all $z\in B$.
Thus, this integral contributes $O_{B,k}(1)$.

As for the main term, for each $j\ge-2$, Mellin inversion gives
$$
\frac1{2\pi i}\int_{\Re(s)=\sigma}L_F(s)(-2\pi iz)^{-s}\Gamma(s+j)\,ds
=\sum_{n=1}^{\infty}a_F(n)(-2\pi inz)^je(nz)
=(-z)^jS_F^{(j)}(z).
$$
Since $\sum_{n=1}^\infty|a_F(n)|n^{-\sigma}$ converges for every
$\sigma>1$, we see that
$z^{-2}S_F^{(-2)}(z)\ll_B 1$
and $z^{-1}S_F^{(-1)}(z)\ll_{B,\varepsilon}\Im(z)^{-\varepsilon}$.
Thus, altogether we have
$$
\frac1{2\pi i}\int_{\Re(s)=\sigma}
\Lambda_F(s)G\!\left(s+\frac{2k}{2-d}\right)(-iz)^{-s}\,ds
=O_{B,k,\varepsilon}(\Im(z)^{-\varepsilon})
+\sum_{j=0}^k(-1)^j\gamma_{kj}z^jS_F^{(j)}(z)
$$
for $z\in B$.
Note that this system of equations is triangular, with non-zero diagonal
coefficients $(-1)^k\gamma_{kk}$.
The lemma follows on multiplying by the inverse matrix.
\end{proof}

For integers $k,\ell\ge0$ we define
$\delta_k=\frac{2k}{2-d}$ and
$\sigma_\ell=\frac12+\frac{2\ell-\Re(\mu)}{2-d}$.
Assume that $\ell$ is such that
$\sigma_\ell>\max\bigl(1,\frac12+\frac{\Re(\mu)-1}{2-d}\bigr)$;
then all poles of $G(s+\delta_k)$ lie to the left of $\Re(s)=\sigma_\ell$
and avoid the line $\Re(s)=1-\sigma_\ell$.
Since $F$ is positive, there is a number $T\in\R$ such that
$\Lambda_F(s)$ is holomorphic for $\Im(s)\ge T$.
Fix such a $T$ and let $\Gamma_\ell$ be the boundary of the region
$\{s\in\C:\Re(s)\in[1-\sigma_\ell,\sigma_\ell], \Im(s)\le T\}$,
with counterclockwise orientation. For $z\in B$,
$|z|\asymp_B 1$ and $\arg{z}-\frac{\pi}2\gg_B 1$, so
it follows from the identity
$$|(-iz)^{-s}|=|z|^{-\Re(s)}\exp\bigl((\arg{z}-\tfrac\pi2)\Im(s)\bigr),$$
the functional equation
$\Lambda_F(s)=\overline{\Lambda_F(1-\bar{s})}$ and Stirling's formula
that
$$
\Lambda_F(s)G(s+\delta_k)(-iz)^{-s}\ll_{B,k,\ell} e^{\frac{\pi}2\Im(s)}
\quad\text{for }s\in \Gamma_\ell,
$$
with an implied constant that is independent of $z$.
Thus, we have
\begin{equation}\label{eq:mainint}
\begin{aligned}
\frac1{2\pi i}\int_{\Re(s)=\sigma_\ell}&\Lambda_F(s)G(s+\delta_k)(-iz)^{-s}\,ds\\
&=\frac1{2\pi i}\left(\int_{\Gamma_\ell}+\int_{\Re(s)=1-\sigma_\ell}\right)
\Lambda_F(s)G(s+\delta_k)(-iz)^{-s}\,ds\\
&=O_{B,k,\ell}(1)+\frac1{2\pi i}\int_{\Re(s)=1-\sigma_\ell}
\Lambda_F(s)G(s+\delta_k)(-iz)^{-s}\,ds\\
&=O_{B,k,\ell}(1)+\frac1{2\pi i}\int_{\Re(s)=\sigma_\ell}\overline{\Lambda_F(\bar{s})}
G(1-s+\delta_k)(-iz)^{s-1}\,ds,
\end{aligned}
\end{equation}
where the last line follows by the functional equation.

\subsection{Degree $<1$}
Let us first see how to use this to find all elements of $\LL_d^+$ for
$d<1$.  Let notation be as in \eqref{eq:mainint} above.
Then for $s=\sigma_\ell+it$, by the bound $L_F(\bar{s})\ll_\ell 1$ and
Proposition~\ref{prop:acfe}(ii), we have
$$
\overline{\Lambda_F(\bar{s})}\ll_\ell
(1+|t|)^{\frac{d}2(\sigma_\ell-\frac12)+\frac{\Re(\mu)}2}
e^{-\frac{\pi}4d|t|}.
$$
On the other hand, Stirling's formula implies that
$$
G(1-s+\delta_k)\ll_{k,\ell}
(1+|t|)^{(\frac{d}2-1)(\sigma_\ell-\frac12)+k-\frac{\Re(\mu)}2}
e^{-\frac{\pi}4(2-d)|t|},
$$
and together these estimates yield
$$
\overline{\Lambda_F(\bar{s})}G(1-s+\delta_k)(-iz)^{s-1}
\ll_{k,\ell}|z|^{\sigma_\ell-1}(1+|t|)^{(d-1)(\sigma_\ell-\frac12)+k}.
$$
Since $d<1$, taking $\ell$ (and hence $\sigma_\ell$) sufficiently large,
we thus have
$$
\frac1{2\pi i}\int_{\Re(s)=\sigma_\ell}\Lambda_F(s)G(s+\delta_k)(-iz)^{-s}\,ds
\ll_{B,k}1.
$$

By Lemma~\ref{lem:SFkint}, it follows that
$S_F^{(k)}(z)\ll_{B,k,\varepsilon}\Im(z)^{-\varepsilon}$ for any
$k\ge0$. Fixing
$$
B=\{-x+iy:x\in[1,2], y\in(0,1]\},
$$
for any positive integer $n$ and $y\in(0,1]$, we have
$$
(2\pi in)^ka_F(n)e^{-2\pi ny}=\int_{-2}^{-1} S_F^{(k)}(x+iy)e(-nx)\,dx
\ll_{k,\varepsilon}y^{-\varepsilon}.
$$
Taking $y=1/n$, we find $a_F(n)\ll_{k,\varepsilon}n^{-k+\varepsilon}$.
In particular, with $k=1$ we see that
$\sum_{n=1}^\infty|a_F(n)|/\sqrt{n}<\infty$,
and thus Theorem~\ref{thm:multone}(iii) implies that $F=(0,0,0)$.

\subsection{Degree $1$}
Henceforth we assume that $d\ge1$.  Note that we are free to shift the
$L$-function of $F$ by an imaginary displacement; precisely, for any
$t\in\R$ one can see directly from Definition~\ref{def:Ldatum} that there
exists $F_t\in\LL_{d}^+$ with $L$-function $L_{F_t}(s)=L_F(s+it)$.  By the
uniqueness of $\gamma$-factors we have $\gamma_{F_t}(s)=\gamma_F(s+it)$,
and in particular the constant $\mu$ given by Proposition~\ref{prop:acfe}
changes to $\mu_t=\mu+idt$. Hence, replacing $F$ by $F_t$ for a suitable
$t$, we may assume without loss of generality that $\mu\in\R$.

With this convention, after a computation similar to that preceding
\eqref{eq:94contour}, using also the identity
$\Gamma(s)\Gamma(1-s)=\frac{\pi}{\sin(\pi{s})}$, we find that
\begin{equation}\label{eq:secformula}
\overline{\gamma_F(\bar{s})}G(1-s)
=A^{s-\frac12}\frac{\Gamma\bigl((d-1)(s-\frac12)+\frac12\bigr)}
{\cos\frac{\pi}2\bigl((2-d)(s-\frac12)+\mu\bigr)}
\left(\sum_{j=0}^{n-1}\frac{\alpha_j}{s^j}+O_n\bigl(|s|^{-n}\bigr)\right)
\end{equation}
for some constants $A\in\R_{>0}$ and $\alpha_j\in\C$
with $\alpha_0\ne0$, uniformly on
$$
\{s\in\C:\Re(s)\ge\tfrac12\}\cap
\{s\in\C:\Re(s)\ge\tfrac12-\tfrac{\mu}{2-d}
\text{ or }|\Im(s)|\ge1\}.
$$

For $d=1$, we fix any permissible value of $\ell$
and substitute \eqref{eq:secformula} into \eqref{eq:mainint} and
Lemma~\ref{lem:SFkint} with $k=0$, obtaining
$$
S_F(z)=
O_{B,\varepsilon}\bigl(\Im(z)^{-\varepsilon}\bigr)
+\frac{c}{2\pi i}\int_{\Re(s)=\sigma_\ell}
\frac{(-iAz)^{s-1}}{\cos\frac{\pi}2\bigl(s-\frac12+\mu\bigr)}
\overline{L_F(\bar{s})}\bigl(1+O(|s|^{-1})\bigr)\,ds
$$
for some constant $c\in\C^\times$.  Using the estimates
$$
\frac{(-iAz)^{s-1}}{\cos\frac{\pi}2\bigl(s-\frac12+\mu\bigr)}
\ll_B e^{-(\pi-\arg{z})|\Im(s)|}
\quad\text{and}\quad\overline{L_F(\bar{s})}\ll 1
$$
for $\Re(s)=\sigma_\ell$,
we see that the $O(|s|^{-1})$ error term contributes
$\ll_B\log\frac{\pi}{\pi-\arg{z}}\ll_{B,\varepsilon}\Im(z)^{-\varepsilon}$.
As for the main term, recalling that $\sigma_\ell=\frac12+2\ell-\mu$, we
have
\begin{align*}
S_F(z)&=
O_{B,\varepsilon}\bigl(\Im(z)^{-\varepsilon}\bigr)
-c(-1)^\ell\sum_{n=1}^{\infty}\frac{\overline{a_F(n)}}{n}
\frac1{2\pi}\int_\R
\frac{(-iAz/n)^{2\ell-\frac12-\mu+it}}{\cosh\frac{\pi{t}}2}\,dt\\
&=O_{B,\varepsilon}\bigl(\Im(z)^{-\varepsilon}\bigr)
-\frac{2ic(-1)^\ell}{\pi}\sum_{n=1}^{\infty}\frac{\overline{a_F(n)}}{n}
\frac{(-iAz/n)^{2\ell-\frac12-\mu}}{\frac{Az}n-\frac{n}{Az}}.
\end{align*}

We now fix $B=\bigl\{-\frac{\alpha-iy}{A}:\alpha\in[1,N+A], y\in(0,1]\}$
for a large integer $N>0$,
and set $z=-\frac{\alpha-iy}{A}$ in the above. If $\alpha$
is not an integer then we see that the last line above is
$O_{\alpha,N,\varepsilon}(y^{-\varepsilon})$, while if 
$\alpha=n\in\Z_{>0}$ then we get
$$
\frac{ce^{i\frac{\pi}2(\frac12-\mu)}}{\pi{y}}\overline{a_F(n)}
+O_{N,\varepsilon}(y^{-\varepsilon}).
$$
Since $S_F(z)$ is periodic and $a_F(1)=1$, it follows that $A$ is an integer
and $a_F(n)=a_F(n+A)$ for all $n\le N$. Thus, since $N$ is arbitrary,
$a_F(n)$ is periodic with period $A$.

Now, since $L_F(s)$ does not vanish for $\Re(s)>1$, it follows from
\cite[Theorem 4]{sw} that there is a positive integer $q|A$, a primitive
Dirichlet character $\chi\pmod{q}$ and a Dirichlet polynomial $D(s)$
such that $L_F(s)=D(s)L(s,\chi)$. Let $F_\chi\in\LL_1^+$ be the
$L$-datum associated to $\chi$, so that $L_{F-F_\chi}(s)=D(s)$. Then
Theorem~\ref{thm:multone}(iv) implies that $F=F_\chi$, and this
completes the proof for $d=1$.

\subsection{Degree $>1$}
We assume now that $d\in(1,2)$ and follow the method of
Kaczorowski and Perelli \cite{kp53}.
In what follows we write $\kappa=\frac1{d-1}$ and
$$
\sigma_k^*=\inf\left\{\sigma\in\R:
\sum_{n=1}^\infty\frac{|a_F(n)|^k}{n^\sigma}<\infty\right\}
\quad\text{for }k\in\{1,2\}.
$$
Since $F\ne(0,0,0)$, it follows from Theorem~\ref{thm:multone} that
$\sum_{n=1}^\infty|a_F(n)|/\sqrt{n}$ diverges, so
$\sigma_1^*\in[\frac12,1]$. On the other hand,
by the Schwarz inequality, we have
$$
\sum_{n=1}^\infty\frac{|a_F(n)|^2}{n^{2\sigma}}
\le\left(\sum_{n=1}^{\infty}\frac{|a_F(n)|}{n^\sigma}\right)^2
\le\zeta(1+\varepsilon)
\sum_{n=1}^\infty\frac{|a_F(n)|^2}{n^{2\sigma-1-\varepsilon}}
$$
for each $\varepsilon>0$,
and thus $2\sigma_1^*\in[\sigma_2^*,\sigma_2^*+1]$.

Continuing along the same lines as
above, we find the following formula for $S_F^{(k)}(z)$ in this case.
\begin{lemma}\label{lem:SFkdgt1}
Fix a compact interval $I\subseteq(0,\infty)$.
Then for any $k\ge1$, $\alpha\in I$ and $y>0$ sufficiently small,
$$
S_F^{(k)}\!\left(-\frac{\alpha-iy}{A}\right)
=O_{I,k,\varepsilon}\bigl(
y^{\frac12-k-(d-1)(\sigma_1^*-\frac12+\varepsilon)}\bigr)
+\frac{\gamma_k}{y^{k+\frac12}}
\sum_{n=1}^\infty\frac{\overline{a_F(n)}}{\sqrt{n}}
\exp\!\left(i\left(\frac{n}{\alpha}\right)^\kappa\right)
V_k\!\left(\frac{n}{\alpha^{d}y^{1-d}}\right),
$$
where $\gamma_k\in\C^\times$ is a constant and
$V_k(t)=t^{\kappa(k+\frac12)}\exp(-\kappa(t^\kappa-1))$.
\end{lemma}

\begin{proof}
Using \eqref{eq:secformula}, we have
\begin{align*}
\frac1{2\pi i}&\int_{\Re(s)=\sigma_\ell}\overline{\Lambda_F(\bar{s})}
G(1-s+\delta_k)(-iz)^{s-1}\,ds=\\
&\frac1{\sqrt{-iz}}\sum_{n=1}^\infty\frac{\overline{a_F(n)}}{\sqrt{n}}
\frac1{2\pi i}\int_{\Re(s)=\sigma_\ell}
\left(\frac{-iAz}{n}\right)^{s-\frac12}
\frac{\Gamma\bigl((d-1)(s-\frac12)+\frac12\bigr)}
{\cos\bigl[\frac{\pi}2\bigl((2-d)(s-\frac12)+\mu\bigr)\bigr]}\\
&\hspace{5cm}\cdot\frac{G(1-s+\delta_k)}{G(1-s)}
\left(\sum_{j=0}^{n-1}\frac{\alpha_j}{s^j}
+O_n\bigl(|s|^{-n}\bigr)\right)ds.
\end{align*}
Let $\Gamma$ be the boundary curve of
$$
\{s\in\C:\Re(s)\ge1+\varepsilon\}\cap
\{s\in\C:\Re(s)\ge\tfrac12+\max(-\tfrac{\mu}{2-d},\tfrac1{d-1})
\text{ or }|\Im(s)|\ge1\}
$$
for a sufficiently small $\varepsilon>0$.
Choosing $n=k+2$, we have
\begin{align*}
&\left(\frac{-iAz}{n}\right)^{s-\frac12}
\frac{\Gamma\bigl((d-1)(s-\frac12)+\frac12\bigr)}
{\cos\bigl[\frac{\pi}2\bigl((2-d)(s-\frac12)+\mu\bigr)\bigr]}
\frac{G(1-s+\delta_k)}{G(1-s)}
\left(\sum_{j=0}^{n-1}\frac{\alpha_j}{s^j}
+O_n\bigl(|s|^{-n}\bigr)\right)\\
&=O_{B,k}(|s|^{-\frac32})+\sum_{j=-1}^k\beta_{kj}
\left(\frac{-iAz}{n}\right)^{s-\frac12}
\frac{\Gamma\bigl((d-1)(s-\frac12)+j+\frac12\bigr)}
{\cos\bigl[\frac{\pi}2\bigl((2-d)(s-\frac12)+\mu\bigr)\bigr]},
\end{align*}
for some constants $\beta_{kj}\in\C$ with $\beta_{kk}\ne0$,
uniformly for $z\in B$ and $s\in\Gamma$.  Thus,
\begin{align*}
&\frac1{2\pi i}\int_{\Re(s)=\sigma_\ell}\overline{\Lambda_F(\bar{s})}
G(1-s+\delta_k)(-iz)^{s-1}\,ds=\\
&O_{B,k}(1)+\frac1{\sqrt{-iz}}
\sum_{j=-1}^k\beta_{kj}
\sum_{n=1}^\infty\frac{\overline{a_F(n)}}{\sqrt{n}}
\frac1{2\pi i}\int_{\Gamma}
\left(\frac{-iAz}{n}\right)^{s-\frac12}
\frac{\Gamma\bigl((d-1)(s-\frac12)+j+\frac12\bigr)}
{\cos\bigl[\frac{\pi}2\bigl((2-d)(s-\frac12)+\mu\bigr)\bigr]}\,ds.
\end{align*}

Next, from the identity $\sec{s}=2e^{-is}-2e^{-is}\sec{s}$ and
Stirling's formula, we see that we may replace
$1/\cos\bigl[\frac{\pi}2\bigl((2-d)(s-\frac12)+\mu\bigr)\bigr]$
by $2\exp\bigl[-\frac{i\pi}2\bigl((2-d)(s-\frac12)+\mu\bigr)\bigr]$
with an error of $O_{B,k}(1)$, uniformly for $z\in B$. Thus, we get
\begin{equation}\label{eq:gammaint}
\begin{aligned}
O_{B,k}(1)+\frac{2e^{-\frac{i\pi}2\mu}}{\sqrt{-iz}}
\sum_{j=-1}^k\beta_{kj}
\sum_{n=1}^\infty\frac{\overline{a_F(n)}}{\sqrt{n}}
\frac1{2\pi i}\int_{\Gamma}
&\left(\frac{A|z|}{n}\right)^{s-\frac12}
e^{i[\frac{\pi}2(d-1)-(\pi-\arg{z})](s-\frac12)}\\
&\cdot\Gamma\bigl((d-1)(s-\tfrac12)+j+\tfrac12\bigr)\,ds.
\end{aligned}
\end{equation}
Assuming that $\Im(z)$ is small enough that $\pi-\arg{z}<\frac\pi2(d-1)$,
we make the change of variables $s\mapsto\kappa{s}+\frac12$, to get
\begin{align*}
\frac1{2\pi i}\int_{\Gamma}
&\left(\frac{A|z|}{n}\right)^{s-\frac12}
e^{i[\frac{\pi}2(d-1)-(\pi-\arg{z})](s-\frac12)}
\Gamma\bigl((d-1)(s-\tfrac12)+j+\tfrac12\bigr)\,ds\\
&=\frac{\kappa}{2\pi i}\int_{\Re(s)=1}
\left(-i\left(-\frac{n}{Az}\right)^\kappa\right)^{-s}
\Gamma\bigl(s+j+\tfrac12\bigr)\,ds\\
&=\kappa\left(-i\left(-\frac{n}{Az}\right)^\kappa\right)^{j+\frac12}
\exp\!\left(i\left(-\frac{n}{Az}\right)^\kappa\right),
\end{align*}
where all powers are taken with
respect to the principal branch of the logarithm.

Let us now fix
$B=\left\{-\frac{\alpha-iy}{A}:\alpha\in I, y\in(0,\delta]\right\}$
for a sufficiently small $\delta>0$,
and put $z=-\frac{\alpha-iy}{A}$ in the above.
Then
$$
\left(-\frac{n}{Az}\right)^\kappa
=\left(\frac{n}{\alpha}\right)^\kappa
\left(1-\frac{iy}{\alpha}\right)^{-\kappa},
$$
so \eqref{eq:gammaint} becomes
\begin{equation}\label{eq:gammaint2}
\begin{aligned}
O_{I,k}(1)+
\sum_{j=-1}^k
&\frac{2\kappa e^{-\frac{i\pi}2(\mu+j+\frac12)}\beta_{kj}}{\sqrt{-iz}}
\left(1-\frac{iy}{\alpha}\right)^{-\kappa(j+\frac12)}\\
&\cdot\sum_{n=1}^\infty\frac{\overline{a_F(n)}}{\sqrt{n}}
\left(\frac{n}{\alpha}\right)^{\kappa(j+\frac12)}
\exp\!\left(i\left(\frac{n}{\alpha}\right)^\kappa
\left(1-\frac{iy}{\alpha}\right)^{-\kappa}\right).
\end{aligned}
\end{equation}
Since $i\bigl(1-\frac{iy}{\alpha}\bigr)^{-\kappa}
=i-\frac{\kappa{y}}{\alpha}+O_I(y^2)$, for sufficiently small $y$
the $j$th term of \eqref{eq:gammaint2} is bounded above by
$O_{I,k,\varepsilon}\bigl(1+
y^{-j-\frac12-(d-1)(\sigma_1^*-\frac12+\varepsilon)}\bigr)$.

We substitute \eqref{eq:gammaint2} into \eqref{eq:mainint} and
Lemma~\ref{lem:SFkint}, and apply the bound noted above to every term
with $j<k$, obtaining
\begin{align*}
S_F^{(k)}(z)=O_{I,k,\varepsilon}\bigl(
y^{\frac12-k-(d-1)(\sigma_1^*-\frac12+\varepsilon)}\bigr)
&+\frac{\gamma_k e^\kappa}{\alpha^{k+\frac12}}
\left(1-\frac{iy}{\alpha}\right)^{-(\kappa+1)(k+\frac12)}\\
&\cdot\sum_{n=1}^\infty\frac{\overline{a_F(n)}}{\sqrt{n}}
\left(\frac{n}{\alpha}\right)^{\kappa(k+\frac12)}
\exp\!\left(i\left(\frac{n}{\alpha}\right)^\kappa
\left(1-\frac{iy}{\alpha}\right)^{-\kappa}\right)
\end{align*}
for any $k>0$ and some constant $\gamma_k\in\C^\times$.
Finally we replace the factor
$\bigl(1-\frac{iy}{\alpha}\bigr)^{-(\kappa+1)(k+\frac12)}$ by $1$ and
$\exp\!\left(i\left(\frac{n}{\alpha}\right)^\kappa
\bigl(1-\frac{iy}{\alpha}\bigr)^{-\kappa}\right)$
by $\exp\!\left(\left(\frac{n}{\alpha}\right)^\kappa
\bigl(i-\frac{\kappa{y}}{\alpha}\bigr)\right)$,
both of which contribute an error of at most
$O_{I,k,\varepsilon}\bigl(
y^{\frac12-k-(d-1)(\sigma_1^*-\frac12+\varepsilon)}\bigr)$, so that
\begin{align*}
S_F^{(k)}(z)=O_{I,k,\varepsilon}\bigl(
y^{\frac12-k-(d-1)(\sigma_1^*-\frac12+\varepsilon)}\bigr)
+\frac{\gamma_k}{y^{k+\frac12}}
\sum_{n=1}^\infty\frac{\overline{a_F(n)}}{\sqrt{n}}
\exp\!\left(i\left(\frac{n}{\alpha}\right)^\kappa\right)
V_k\!\left(\frac{n}{\alpha^{d}y^{1-d}}\right),
\end{align*}
as desired.
\end{proof}

\begin{lemma}\label{lem:Sigmaformula}
Let $w:\R\to\C$ be a smooth function supported on a compact subinterval of
$(0,\infty)$, and define
$$
\widehat{w}(x)=\int_\R w(t)e(xt)\,dt,
\quad
W(u)=\begin{cases}
u^{\frac\kappa2+\frac1{2d}}(u^{1/d}-1)^{\frac{\kappa-1}2}
w\bigl((u^{1/d}-1)^\kappa\bigr)&\text{if }u>1,\\
0&\text{if }u\le1,
\end{cases}
$$
and
$$
\Sigma(x)=\frac{e^{-\frac{i\pi}4}}{\sqrt{ad}}\sum_{n=1}^{\infty}
\frac{\overline{a_F(n)}}{n^{\frac{\kappa+1}2}}
e\!\left(a\bigl(n^{\frac1{d}}-x^{\frac1{d}}\bigr)^{\kappa d}\right)
W\!\left(\frac{n}{x}\right),
$$
where $a=\frac1{2\pi A^\kappa}$.
Then for all $x>0$ sufficiently large,
\begin{equation}\label{eq:Sigmaformula}
\Sigma(x)=
\frac{\overline{a_F(n(x))}}{\sqrt{n(x)}}
\widehat{w}\bigl(-a\kappa{n(x)}^{\kappa-1}[x-n(x)]\bigr)
+O_{w,\varepsilon}\bigl(x^{\sigma_1^*-\frac12-\kappa+\varepsilon}\bigr),
\end{equation}
where $n(x)=\lfloor{x+\frac12}\rfloor$ is the nearest integer to $x$.
\end{lemma}
\begin{proof}
We apply Lemma~\ref{lem:SFkdgt1} with $k=1$.
Since $S_F'(z)$ is periodic, the formula is invariant under
$\alpha\mapsto \alpha+A$, so that
\begin{align*}
&\sum_{n=1}^\infty\frac{\overline{a_F(n)}}{\sqrt{n}}
\exp\!\left(i\left(\frac{n}{\alpha}\right)^\kappa\right)
V_1\!\left(\frac{n}{\alpha^{d}y^{1-d}}\right)
=O_{I,\varepsilon}\bigl(
y^{1-(d-1)(\sigma_1^*-\frac12+\varepsilon)}\bigr)\\
&+\sum_{n=1}^\infty\frac{\overline{a_F(n)}}{\sqrt{n}}
\exp\!\left(i\left(\frac{n}{\alpha+A}\right)^\kappa\right)
V_1\!\left(\frac{n}{(\alpha+A)^{d}y^{1-d}}\right).
\end{align*}
Introducing new parameters $x=\alpha^{d}y^{1-d}$
and $t=(A/\alpha)^\kappa$, this becomes
$$
\sum_{n=1}^\infty\frac{\overline{a_F(n)}}{\sqrt{n}}
e(an^\kappa{t})V_1\!\left(\frac{n}{x}\right)
=O_{I,\varepsilon}\bigl(x^{\sigma_1^*-\frac12-\kappa+\varepsilon}\bigr)
+\sum_{n=1}^\infty\frac{\overline{a_F(n)}}{\sqrt{n}}
e\!\left(\frac{an^\kappa{t}}{(1+t^{d-1})^\kappa}\right)
V_1\!\left(\frac{n}{x(1+t^{d-1})^{d}}\right).
$$

Now let $w:\R\to\C$ be as in the hypotheses, and
fix the interval $I$ in Lemma~\ref{lem:SFkdgt1} so that
it contains $At^{1-d}$ for every $t$ in the support of $w$.
We multiply both sides of the above by
$e(-at x^\kappa)w(t)$ and integrate to get
\begin{equation}\label{eq:mainid}
\begin{aligned}
&\sum_{n=1}^\infty\frac{\overline{a_F(n)}}{\sqrt{n}}
\widehat{w}\bigl(a(n^\kappa-x^\kappa)\bigr)V_1\!\left(\frac{n}{x}\right)
=O_{w,\varepsilon}\bigl(x^{\sigma_1^*-\frac12-\kappa+\varepsilon}\bigr)\\
&+\sum_{n=1}^\infty\frac{\overline{a_F(n)}}{\sqrt{n}}
\int_0^\infty e\bigl(x^\kappa\varphi(t,n/x)\bigr)\psi(t,n/x)\,dt,
\end{aligned}
\end{equation}
where
$$
\varphi(t,u)
=at\left[\left(\frac{u}{1+t^{d-1}}\right)^\kappa-1\right]
\quad\text{and}\quad
\psi(t,u)=w(t)V_1\!\left(\frac{u}{(1+t^{d-1})^{d}}\right).
$$

Suppose that $w$ is supported on $[t_1,t_2]\subseteq(0,\infty)$,
and set $\lambda_1=\frac12\bigl(1+(1+t_1^{d-1})^d\bigr)$,
$\lambda_2=2(1+t_2^{d-1})^d$.
To treat the right-hand side of \eqref{eq:mainid}, we split the sum over
the three ranges $n\le\lambda_1x$, $\lambda_1x<n<\lambda_2x$ and
$n\ge\lambda_2x$.  We compute that
$$
\frac{\partial\varphi}{\partial{t}}(t,u)
=a\left[\frac{u^\kappa}{(1+t^{d-1})^{\kappa{d}}}-1\right]
\quad\text{and}\quad
\frac{\partial^2\varphi}{\partial{t}^2}(t,u)
=-\frac{adu^{\kappa}t^{d-2}}
{\bigl(1+t^{d-1}\bigr)^{\kappa d+1}},
$$
so that $\frac{\partial\varphi}{\partial{t}}$
vanishes only when $u=(1+t^{d-1})^{d}$.
Hence, for $u\le\lambda_1$ or $u\ge\lambda_2$,
$\frac{\partial\varphi}{\partial{t}}$ does not vanish
for $t\in[t_1,t_2]$, and in fact we have
$|\frac{\partial\varphi}{\partial{t}}|\gg_w1$
uniformly for $t\in[t_1,t_2]$,
$u\notin(\lambda_1,\lambda_2)$.  Further,
$\frac{\partial^2\varphi}{\partial{t}^2}$ never vanishes, so
by van der Corput's lemma
\cite[Chapter VIII, Corollary of Proposition 2]{stein}, we have
$$
\int_0^\infty e\bigl(x^{\kappa}\varphi(t,u)\bigr)\psi(t,u)\,dt
\ll_w x^{-\kappa},
$$
uniformly for $u\le\lambda_1$ or $u\ge\lambda_2$.
Thus, the ranges $n\le\lambda_1x$ and $n\ge\lambda_2x$ contribute
$O_{w,\varepsilon}\bigl(x^{\sigma_1^*-\frac12-\kappa+\varepsilon}\bigr)$.

For the remaining range we apply the method of stationary phase
\cite[Lemma 2.8]{tao}.  Note in particular that
$$
\frac{\partial^2\varphi}{\partial{t}^2}(t_0(u),u)=
-\frac{ad}{u^{1/d}(u^{1/d}-1)^{\kappa-1}}
\quad\text{and}\quad
\varphi(t_0(u),u) = a\bigl(u^{1/d}-1\bigr)^{\kappa d},
$$
where $t_0(u)=(u^{1/d}-1)^\kappa$ is the stationary point of
$\varphi(t,u)$. Since $u=n/x$ varies within a compact subset of
$(1,\infty)$, we find that
$$
\int_0^\infty e\bigl(x^{\kappa}\varphi(t,u)\bigr)\psi(t,u)\,dt
=\frac{e^{-\frac{i\pi}4}}{\sqrt{ad}}n^{-\frac{\kappa}2}
e\!\left(a\bigl(n^{1/d}-x^{1/d}\bigr)^{\kappa d}\right)
W\!\left(\frac{n}{x}\right)
+O_w\bigl(x^{-\frac32\kappa}\bigr),
$$
with $W$ as defined in statement of the lemma, and
with an implied constant that is independent of $n$.
The error term contributes a total of at most
$O_{w,\varepsilon}(x^{\sigma_1^*-\frac12-\frac32\kappa+\varepsilon})$,
so altogether the right-hand side of \eqref{eq:mainid} is
$$
O_{w,\varepsilon}\bigl(x^{\sigma_1^*-\frac12-\kappa+\varepsilon}\bigr)
+\frac{e^{-\frac{i\pi}4}}{\sqrt{ad}}
\sum_{n=1}^\infty\frac{\overline{a_F(n)}}{n^{\frac{\kappa+1}2}}
e\!\left(a\bigl(n^{1/d}-x^{1/d}\bigr)^{\kappa d}\right)
W\!\left(\frac{n}{x}\right).
$$

Finally, we note that the left-hand side of \eqref{eq:mainid} is essentially
concentrated at integers. Precisely,
for any $n\ne n(x)=\lfloor{x+\frac12}\rfloor$, we have
$|n^\kappa-x^\kappa|\gg x^{\kappa-1}$. Since $\widehat{w}$ has very
rapid decay, these terms contribute $O_N(x^{-N})$ to the sum.
On the other hand, for $n=n(x)$, writing $\{x\}=x-n(x)$, we have
$$
n^\kappa-x^\kappa=n^\kappa-(n+\{x\})^\kappa
=-\kappa n^{\kappa-1}\{x\}+O\bigl(x^{\kappa-2}\{x\}^2\bigr),
$$
and it follows that
$$
\frac{\overline{a_F(n)}}{\sqrt{n}}
\widehat{w}\bigl(a(n^\kappa-x^\kappa)\bigr)V_1\!\left(\frac{n}{x}\right)
=\frac{\overline{a_F(n)}}{\sqrt{n}}\bigl(
\widehat{w}\bigl(-a\kappa n^{\kappa-1}\{x\}\bigr)
+O_\varepsilon(x^{-\kappa+\varepsilon})\bigr).
$$
Noting that $a_F(n)\ll_\varepsilon n^{\frac12+\varepsilon}$,
this concludes the proof.
\end{proof}

To apply Lemma~\ref{lem:Sigmaformula}, we fix a function
$w_0:\R\to\R$ which is smooth, even, non-negative, supported on
$[-\frac12,\frac12]$ and $L^2$-normalized, and set
$w(t)=w_0(t-\frac32)$. Since the corresponding $W$ is supported
away from $0$, from the definition of $\Sigma(x)$ we get
$$
|\Sigma(x)|\le\sum_{n=1}^\infty\frac{|a_F(n)|}{n^{\frac{\kappa+1}2}}
\left|W\!\left(\frac{n}{x}\right)\right|
\ll_\varepsilon x^{\sigma_1^*-\frac{\kappa+1}2+\varepsilon}.
$$
Setting $x=n$, we learn from \eqref{eq:Sigmaformula} that
$a_F(n)\ll_\varepsilon n^{\sigma_1^*-\frac{\kappa}2+\varepsilon}$.
However, if $\kappa>2$, this contradicts the definition of $\sigma_1^*$,
and thus $\LL_d^+=\emptyset$ for $d\in(1,\frac32)$.

The idea of \cite{kp53} for going beyond this is to exploit
the fact that the right-hand side of \eqref{eq:Sigmaformula}
is concentrated at integers, so that the integral
$J(X)=\int_X^{2X}|\Sigma(x)|^2e(x)\,dx$ behaves like
$\int_X^{2X}|\Sigma(x)|^2\,dx$. Since the series
defining $\Sigma(x)$ displays no such behavior, we will see that $J(X)$
is small, and this results in a contradiction for $d<\frac53$.
We assume henceforth that $d\in[\frac32,2)$, so that $\kappa\in(1,2]$.

Proceeding, we expand the square on the right-hand side of
\eqref{eq:Sigmaformula}.
The main term is
\begin{align*}
\int_X^{2X}&\frac{|a_F(n(x))|^2}{n(x)}
\bigl|\widehat{w}\bigl(-a\kappa{n(x)}^{\kappa-1}[x-n(x)]\bigr)\bigr|^2
e(x)\,dx\\
&=\sum_{X-\frac12\le n\le 2X+\frac12}\frac{|a_F(n)|^2}{n}
\int_{[n-\frac12,n+\frac12)\cap[X,2X]}
\bigl|\widehat{w}\bigl(-a\kappa{n}^{\kappa-1}(x-n)\bigr)\bigr|^2
e(x)\,dx.
\end{align*}
For the boundary terms with $n$ near $X$ or $2X$, we use the above estimate
$$
|a_F(n)|^2\ll_\varepsilon n^{2\sigma_1^*-\kappa+\varepsilon}
\le n^{\sigma_2^*+1-\kappa+\varepsilon}
$$
to see that they contribute at most
$$
\frac{|a_F(n)|^2}{n}\int_\R
\bigl|\widehat{w}\bigl(-a\kappa{n}^{\kappa-1}x\bigr)\bigr|^2\,dx
\ll_\varepsilon X^{\sigma_2^*+1-2\kappa+\varepsilon}.
$$
For the other terms we translate the integral by $n$ and extend it to
$\R$, which introduces an
error of only $O_N(X^{-N})$ thanks to the rapid decay of $\widehat{w}$.
Thus, in total the main term is
\begin{align*}
&O_\varepsilon\bigl(X^{\sigma_2^*+1-2\kappa+\varepsilon}\bigr)
+\frac1{a\kappa}\sum_{X\le n\le 2X}\frac{|a_F(n)|^2}{n^\kappa}
w_0\ast w_0\!\left(\frac1{a\kappa n^{\kappa-1}}\right)\\
&=O_\varepsilon\bigl(X^{\sigma_2^*+1-2\kappa+\varepsilon}\bigr)
+\frac1{a\kappa}\sum_{X\le n\le 2X}\frac{|a_F(n)|^2}{n^\kappa},
\end{align*}
where we have used that $w_0\ast w_0(t)=1+O(|t|)$.
Next, the error terms in \eqref{eq:Sigmaformula} contribute
\begin{align*}
&\ll_\varepsilon X^{2\sigma_1^*-2\kappa+\varepsilon}
+ X^{\sigma_1^*-\frac12-\kappa+\varepsilon}
\int_X^{2X}
\left|\frac{\overline{a_F(n(x))}}{\sqrt{n(x)}}
\widehat{w}\bigl(-a\kappa{n(x)}^{\kappa-1}[x-n(x)]\bigr)\right|\,dx\\
&\ll_\varepsilon X^{2\sigma_1^*-2\kappa+\varepsilon}
\le X^{\sigma_2^*+1-2\kappa+\varepsilon},
\end{align*}
so altogether we have
\begin{equation}\label{eq:Jformula}
J(X)=\frac1{a\kappa}\sum_{X\le n\le 2X}\frac{|a_F(n)|^2}{n^\kappa}
+O_\varepsilon\bigl(X^{\sigma_2^*+1-2\kappa+\varepsilon}\bigr).
\end{equation}

Now, for any fixed $\varepsilon>0$, since
$\sum_{n=1}^{\infty}|a_F(n)|^2n^{-\sigma_2^*+\frac{\varepsilon}2}$
diverges, there are arbitrarily large values of $X$ such that
$\sum_{X\le n\le 2X}|a_F(n)|^2n^{-\sigma_2^*+\frac{\varepsilon}2}
\ge\frac1{\log{X}}$,
and for these $X$ we have
$$
\sum_{X\le n\le 2X}\frac{|a_F(n)|^2}{n^\kappa}
\gg_\varepsilon X^{\sigma_2^*-\kappa-\frac{\varepsilon}2}
\sum_{X\le n\le 2X}\frac{|a_F(n)|^2}{n^{\sigma_2^*-\frac{\varepsilon}2}}
\gg_\varepsilon X^{\sigma_2^*-\kappa-\varepsilon}.
$$
Hence, since $\kappa>1$, \eqref{eq:Jformula} implies that
there are arbitrarily large $X$ for which
\begin{equation}\label{eq:Jlowerbound}
|J(X)|\gg_\varepsilon X^{\sigma_2^*-\kappa-\varepsilon}.
\end{equation}

Next we evaluate $J(X)$ using the definition of $\Sigma(x)$. To that
end, since our chosen $w$ is real valued, we have
$$
|\Sigma(x)|^2
=\frac1{ad}\sum_{m,n\ge1}
\frac{a_F(m)\overline{a_F(n)}}{(mn)^{\frac{\kappa+1}2}}
e\!\left(a\bigl(n^{\frac1{d}}-x^{\frac1{d}}\bigr)^{\kappa d}
-a\bigl(m^{\frac1{d}}-x^{\frac1{d}}\bigr)^{\kappa d}\right)
W\!\left(\frac{m}{x}\right)W\!\left(\frac{n}{x}\right).
$$
We are thus faced with the integral
$$
J_{m,n}(X)=
\int_X^{2X}e\bigl(f(x,n)-f(x,m)+x\bigr)
W\!\left(\frac{m}{x}\right)W\!\left(\frac{n}{x}\right)dx,
$$
where $f(x,n)=a\bigl(n^{\frac1{d}}-x^{\frac1{d}}\bigr)^{\kappa d}$, to
which we apply van der Corput's method
\cite[Chapter VIII, Corollary of Proposition 2]{stein}.

First note that
$$
\frac{\partial}{\partial{x}}\bigl(f(x,n)-f(x,m)\bigr)
=-a\kappa x^{\kappa-1}\bigl(t_0(n/x)-t_0(m/x)\bigr),
$$
where $t_0(u)=(u^{1/d}-1)^\kappa$. For $u$ in the support of $W$, we
have
$t_0'(u)=\frac{\kappa}{d}(u^{1/d}-1)^{\kappa-1}u^{\frac1d-1}\asymp 1$,
so by the mean value theorem, there are positive
constants $c_1$ and $c_2$ such that
$$
c_1x^{\kappa-2}\le
\frac{\frac{\partial}{\partial{x}}\bigl(f(x,n)-f(x,m)\bigr)}{m-n}
\le c_2x^{\kappa-2}
$$
for all $m\ne n$ such that $W(m/x)W(n/x)\ne0$ for some $x\in[X,2X]$.

Put
$I_X=\bigl[\frac{X^{2-\kappa}}{2c_2},\frac{2(2X)^{2-\kappa}}{c_1}\bigr]$.
Then for $0\ne n-m\notin I_X$, it follows that
$$
\left|\frac{\partial}{\partial{x}}\bigl(f(x,n)-f(x,m)+x\bigr)\right|
\ge\frac{c_1}2x^{\kappa-2}|m-n|\gg X^{\kappa-2}|m-n|.
$$
Further, we compute that
\begin{equation}\label{eq:2ndderiv}
\begin{aligned}
\frac{\partial^2}{\partial{x}^2}\bigl(f(x,n)-f(x,m)+x\bigr)
&=a\kappa x^{\kappa-2}
\bigl[ut_0'(u)-(\kappa-1)t_0(u)\bigr]\Bigr|_{m/x}^{n/x}\\
&=ax^{\kappa-2}
\bigl(u^{1/d}-1\bigr)^{\kappa-1}
\bigl[\tfrac1d u^{1/d}+\kappa(\kappa-1)\bigr]\Bigr|_{m/x}^{n/x}.
\end{aligned}
\end{equation}
Since
$\bigl(u^{1/d}-1\bigr)^{\kappa-1}
\bigl[\tfrac1d u^{1/d}+\kappa(\kappa-1)\bigr]$
is an increasing function, the last line never vanishes
for $m\ne n$, so
$\frac{\partial}{\partial{x}}\bigl(f(x,n)-f(x,m)+x\bigr)$ is monotonic.
Thus, van der Corput's lemma for the first derivative yields the
estimate
$$
J_{m,n}(X)\ll\frac{X^{2-\kappa}}{|m-n|}
$$
for those terms.  Similarly, for the diagonal terms $m=n$, we get
$$
J_{m,n}(X)=\int_X^{2X}W\!\left(\frac{n}{x}\right)^2 e(x)\,dx\ll 1.
$$
It remains only to handle the terms with
$n-m\in I_X$, to which we apply van der
Corput's lemma for the second derivative.
From \eqref{eq:2ndderiv} and the mean value theorem, we see that
$$
\left|\frac{\partial^2}{\partial{x}^2}\bigl(f(x,n)-f(x,m)+x\bigr)\right|
\gg x^{\kappa-3}|m-n| \gg X^{-1},
$$
and thus $J_{m,n}(X)\ll\sqrt{X}$.

Suppose that $W$ is supported on $[u_1,u_2]\subseteq(0,\infty)$.
Substituting the above estimates into the definition of $J(X)$, we have
$$
\begin{aligned}
J(X)\ll \sum_{u_1X\le n\le2u_2 X}
\frac{|a_F(n)|^2}{n^{\kappa+1}}
&+\sqrt{X}\sum_{\substack{u_1X\le m,n\le2u_2X\\
n-m\in I_X}}\frac{|a_F(n)a_F(m)|}{(mn)^{\frac{\kappa+1}2}}\\
&+X^{2-\kappa}\sum_{\substack{u_1X\le m,n\le2u_2X\\
n-m\notin I_X\cup\{0\}}}\frac{|a_F(n)a_F(m)|}{|m-n|(mn)^{\frac{\kappa+1}2}}.
\end{aligned}
$$
Using the inequality $|a_F(n)a_F(m)|\le\frac12(|a_F(n)|^2+|a_F(m)|^2)$
together with the estimates $\#(I_X\cap\Z)\ll X^{2-\kappa}$ and
$\sum_{1\le n\le2u_2X}\frac1n\ll\log{X}\ll_\varepsilon X^\varepsilon$,
we see that this is
$$
\ll_\varepsilon X^{\sigma_2^*-1-\kappa+\varepsilon}
+ X^{\sigma_2^*+\frac32-2\kappa+\varepsilon}
+ X^{\sigma_2^*+1-2\kappa+\varepsilon}
\ll X^{\sigma_2^*+\frac32-2\kappa+\varepsilon}.
$$
Putting this together with the lower bound \eqref{eq:Jlowerbound},
we must have $\sigma_2^*-\kappa-\varepsilon\le
\sigma_2^*+\frac32-2\kappa+\varepsilon$
for all $\varepsilon>0$, and thus $\kappa\le\frac32$.
Hence, $\LL_d^+=\emptyset$ for $d\in[\frac32,\frac53)$, and this
concludes the proof.

\bibliographystyle{amsplain}
\bibliography{selberg}
\end{document}